	\documentclass[12pt,a4paper]{amsart}
	\usepackage{amsmath,amssymb}
	\usepackage{amsthm,graphicx, color}
  \usepackage{enumerate}
	\usepackage[T1]{fontenc}
	\usepackage{float}

	\numberwithin{equation}{section}

	\DeclareMathOperator{\inv}{inv}

	\newtheorem{thm}{Theorem}[section]
	\newtheorem{lemma}[thm]{Lemma}
	
	\newtheorem{coroll}[thm]{Corollary}

	\newtheorem{problem}[thm]{Problem}
	\theoremstyle{definition}
	\newtheorem{defi}{Definition}[section]
	\newtheorem{example}{Example}[section]
	\newtheorem*{remark}{Remark}

	\newcommand{\Sym}{\mathfrak{S}}
	\newcommand{\SymJ}{\mathfrak{J}}
	\newcommand{\SymK}{\mathfrak{K}}
	\newcommand{\SymA}{\mathfrak{A}}
	\newcommand{\SymP}{\mathfrak{P}}

	\newcommand{\String}[1]{\text{`$#1$'}}
	
\newcommand{\ch}[1]{\hbox to 0.9em {\hfil $#1$\hfil}} 
\newcommand{\cch}[1]{\hbox to 1.1em {\hfil $#1$\hfil}}
\newcommand{\ccch}[1]{\hbox to 2.2em {\hfil $#1$\hfil}}
\newcommand{\block}[2]{\begin{matrix}#1\\#2\end{matrix}}

\newcommand{\spred}{\kern -0.1em}

\usepackage{float}

\floatstyle{ruled}
\newfloat{program}{thp}{lop}
\floatname{program}{Algorithm}

\usepackage{chngcntr}
\counterwithin{table}{section}

\allowdisplaybreaks
	\title[Computer assisted proof for Apwenian sequences]{Computer assisted proof for Apwenian sequences related to {H}ankel determinants} 
	\date{May 21, 2015}
	\author{Hao Fu}
	\address{Institute for Interdisciplinary Information Sciences\\
	Tsinghua University\\
	Beijing,100084\\
	P.R.China}
	\email{fu-h13@mails.tsinghua.edu.cn}
  \thanks{The first author was supported by the National Basic Research Program of China Grant 2011CBA00300, 2011CBA00301, the National Natural Science Foundation of China Grant 61033001, 61361136003.}

	\author{Guo-Niu HAN}
	\address{Institut de Recherche Math\'ematique Avanc\'ee\\
	Universit\'e de Strasbourg et CNRS\\
7 rue Ren\'e Descartes\\
	 67084 Strasbourg\\
	France} \email{guoniu.han@unistra.fr}

\begin{document}
	\begin{abstract}
		An infinite $\pm 1$-sequence is called {\it Apwenian} if its Hankel determinant of order $n$ divided by $2^{n-1}$ is an odd number for every positive integer $n$.
	In 1998, Allouche, Peyri\`ere, Wen and Wen discovered and 
	proved that
	the Thue--Morse sequence is an Apwenian sequence by direct determinant manipulations.
	Recently, Bugeaud and Han re-proved the latter result by means of an appropriate combinatorial method. By significantly improving the combinatorial method, 
	we prove that several other  Apwenian sequences  related to the Hankel determinants with Computer Assistance.
	\end{abstract}
\maketitle

\section{introduction}
For each infinite sequence
${\bf c}=(c_k)_{k\geq 0}$
and each nonnegative integer~$n$ 
the {\it Hankel determinant} of order $n$ of 
the sequence ${\bf c}$ is defined by
\begin{equation}\label{def:Hankel}
H_{n}({\bf c}):=
\begin{vmatrix}
  c_0 & c_{1}&\cdots & c_{n-1} \\
  c_{1} & c_{2}& \cdots & c_{n}\\
   \vdots & \vdots &\ddots & \vdots\\
   c_{n-1} & c_{n} & \cdots & c_{2n-2}
\end{vmatrix}.
\end{equation}
We also speak of the Hankel determinants of the power series 
$\tilde{\bf c}(x)=\sum_{k\geq 0} c_k x^k$ and write
$H_n(\tilde{\bf c}(x)) = H_n({\bf c})$.
The Hankel determinants are widely studied in Mathematics and,
in several cases, can be evaluated by  
basic determinant manipulation, $LU$-decomposition, or  Jacobi 
continued fraction (see, e.g., \cite{Kr98, Kr05, Fl80, Wa48, Mu23}).
However, the Hankel determinants studied in the present paper
apparently have no closed-form expressions, and require additional efforts
to obtain specific arithmetical properties.

An infinite $\pm 1$-sequence ${\bf c}=(c_k)_{k\geq 0} $ is called {\it Apwenian} if its Hankel determinant of order $n$ divided by $2^{n-1}$ is an odd number, i.e.,
$H_n({\bf c})/2^{n-1}\equiv 1\pmod2$, for all positive integer $n$.
The corresponding generating function or the power series $\tilde{\bf c}(x)$ is also said to be {\it Apwenian}.
Recall that the {\it Thue--Morse sequence}, denoted by
$${\bf e} =  (e_k)_{k\geq 0} = (1,-1,-1,1, -1, 1, 1, -1, -1,1,1,-1\ldots),$$ 
is a special $\pm 1$-sequence \cite{Wiki:pm:sequence},
defined by the
generating function
\begin{equation}\label{def:ThueMorse}
	\tilde{\bf e}(x)=\sum_{k=0}^\infty e_k x^k=\prod_{k=0}^\infty(1-x^{2^k}),
\end{equation}
or equivalently, by the recurrence relations
\begin{equation}\label{def:ThueMorse:Rec}
	e_0=1,\quad e_{2k}=e_k \text{ and } e_{2k+1}=-e_k \text{ for } k\geq 0.
\end{equation}
The Thue--Morse sequence is also called {\it Prouhet--Thue--Morse sequence}.
For other equivalent definitions and properties related to the sequence, see 
\cite{AS2003, AS1999, Wiki:ThueMorse, OEIS:A106400, OEIS:A010060}.
In 1998, 
Allouche, Peyri\`ere, Wen and Wen
established a congruence relation concerning the Hankel determinants
of the Thue--Morse sequence~\cite{APWW1998}.
\begin{thm} [APWW]
\label{thm:APWW}
The Thue--Morse sequence on $\{1, -1\}$ is Apwenian.
\end{thm}
Theorem \ref{thm:APWW} has an important application to Number Theory.
As a consequence of Theorem \ref{thm:APWW}, all the Hankel determinants of the
Thue--Morse sequence are nonzero. This property 
allowed Bugeaud \cite{Bu2011} to prove that the irrationality exponents of the Thue--Morse--Mahler numbers are exactly 2.
\medskip

The goal of the paper is to find more Apwenian sequences.
Let $d$ be a positive integer and ${\bf v}=(v_0, v_1,v_2,\ldots,v_{d-1})$ a finite $\pm 1$-sequence of length $d$ such that $v_0=1$.
The generating polynomial of $\bf v$ is denoted by 
$\tilde{\bf v}(x)= \sum_{i=0}^{d-1} v_i x^i $.
It is clear that the following power series
\begin{equation}\label{def:Phi}
	\Phi(\tilde{\bf v}(x)) = \prod_{k=0}^\infty \tilde{\bf v} (x^{d^k})
\end{equation}
defines a $\pm1$-sequence.
Thus, the power series displayed in \eqref{def:ThueMorse} 
is equal to $\Phi(1-x)$. 
Our main result is stated next.

\begin{thm}\label{thm:main}
The following power series are all Apwenian:
\begin{align*}
F_2(x)	&=\Phi(1-x),\\
F_3(x)	&=\Phi(1-x-x^2),\\
F_5(x)	&=\Phi(1-x-x^2-x^3+x^4),\\
F_{11}(x)	&=\Phi(1-x-x^2+x^3-x^4+x^5+x^6+x^7+x^8-x^9\\
					&\qquad\qquad-x^{10}),\nonumber\\
F_{13}(x)		&=\Phi(1-x-x^2+x^3-x^4-x^5-x^6-x^7-x^8+x^9\\
				 &\qquad\qquad-x^{10}-x^{11}+x^{12}),\nonumber\\
F_{17a}(x)		&=\Phi(1-x-x^2+x^3-x^4+x^5+x^6+x^7+x^8+x^9\\
		&\qquad\qquad+x^{10} +x^{11}-x^{12}+x^{13}-x^{14}-x^{15}+x^{16}),\nonumber\\
F_{17b}(x)		&=\Phi(1-x-x^2-x^3+x^4+x^5-x^6+x^7+x^8+x^9\\
		&\qquad\qquad-x^{10} +x^{11}+x^{12}-x^{13}-x^{14}-x^{15}+x^{16}).\nonumber
\end{align*}
\end{thm}

{\bf Remarks}. Let us make some useful comments about the above Theorem.
\begin{enumerate}
\item
The fact that the generating function $F_2(x)$ for the Thue--Morse 
sequence is Apwenian has already been proved in \cite{APWW1998}.
\item 
	By using the Jacobi continued fraction expansion of a power series $F(x)$, 
	we know that
$H_n(F(x))=H_n(F(-x))$. 
See, for example, \cite{Kr98,Fl80, Han2014, Han2015hankel}. 
Hence, Theorem \ref{thm:main} implies that
$F_3(-x)=\Phi(1+x-x^2)$, $F_5(-x)=\Phi(1+x-x^2+x^3+x^4)$, etc. are all Apwenian.
\item 
	There is no $F_7$ in Theorem \ref{thm:main}, but two $F_{17}$ (we mean
	$F_{17a}$ and $F_{17b}$).
\item
	$\Phi(1-x-x^2+x^3)$ is Apwenian since it is equal to $\Phi(1-x)$.
\end{enumerate}

Actually, Theorem \ref{thm:APWW} has three proofs. The original proof of Theorem \ref{thm:APWW} is based on determinant manipulation by using the so-called {\it sudoku method} \cite{APWW1998, HanWu2015}.
The second one is a combinatorial proof derived by Bugeaud and Han \cite{BH2014}.
The third proof is very short  and makes use of Jacobi continued fraction algebra \cite{Han2015hankel}. Unfortunately, the method developed in the 
short proof cannot be used for proving our main theorem, 
because the underlying Jacobi continued fractions are not ultimately periodic \cite{Han2015hankel, Han2014}.
However, another analogous result for the sequence $F_3(x)$ when dealing with modulo 3 (instead of modulo 2) is established using the short method, as stated in the next theorem~\cite{Han2014}.

\begin{thm}\label{thm:mod3}
	For every positive integer $n$ the Hankel determinant $H_{n}(F_3(x))$ of the sequence $F_3(x)$ verifies the following relation
\begin{equation}\label{equ:mod3}
	H_{n}(F_3(x))\equiv 
\begin{cases}
	1 \pmod 3 &\text{if $n\equiv 1,2\pmod 4$;} \\
	2 \pmod 3 &\text{if $n\equiv 3,0\pmod 4$.} \\
\end{cases}
\end{equation}
\end{thm}

Combining Theorem \ref{thm:mod3} and Theorem \ref{thm:main} yields 
the following result.
\begin{coroll}\label{thm:mod6}
	For every positive integer $n$ the Hankel determinant $H_{n}(F_3(x))$ verifies the following relation
\begin{equation}\label{equ:mod6}
	\frac{H_{n}(F_3(x))}{2^{n-1}}\equiv 
\begin{cases}
	1 \pmod 6 &\text{if $n\equiv 0,1\pmod 4$;} \\
	5 \pmod 6 &\text{if $n\equiv 2,3\pmod 4$.} \\
\end{cases}
\end{equation}
\end{coroll}

In the following table we reproduce the first few values of the Hankel determinants of the sequence $F_3(x)$ for illustrating Theorems \ref{thm:main}, \ref{thm:mod3} and Corollary \ref{thm:mod6}.

\bigskip
\centerline{
\begin{tabular}{|l|r r r r r r r r r r|}
\hline
$n$												& $1$ & $2$		& $3$		& $4$ & $5$ & $6$ 	&	$7$ & $8$ & $9$ & $10$\\
\hline 
$H_n({\bf f})$ 									& $1$ & $-2$	& $-4$	& $8$ & $16$ & $-32$ &	$-64$	& $128$ & $4864$ & $-9728$\\ [2pt]
$H_n({\bf f}) \pmod3$ 					& $1$ & $1$		& $2$		& $2$ & $1$ & $1$		& $2$ & $2$ & $1$ & $1$\\ [2pt]
$H_n({\bf f})/2^{n-1}$ 									& $1$ & $-1$	& $-1$	& $1$ & $1$ & $-1$ &	$-1$	& $1$ & $19$ & $-19$\\ [2pt]
$\frac{H_n({\bf f})}{2^{n-1}} \pmod 2$	& $1$	& $1$	& $1$		& $1$		& $1$ & $1$ & $1$ & $1$ & $1$ & $1$	\\	[4pt]
$\frac{H_n({\bf f})}{2^{n-1}} \pmod6$ 					& $1$ & $5$		& $5$		& $1$ & $1$ & $5$		& $5$ & $1$ & $1$ & $5$\\ [2pt]
\hline
\end{tabular}
}

\bigskip

Recently, Bugeaud and Han re-proved Theorem \ref{thm:APWW} by means of an appropriate combinatorial method \cite{BH2014}.  
The latter method 
has been significantly upgraded 
to prove that $F_3(x)$ is Apwenian.
As can be seen, in Section \ref{sec:CountType} Step 2,
a family of cases (called {\it types}) is considered for proving 
the various recurrence relations.
Roughly speaking, the types are indexed by words $s_0s_1s_2\cdots s_d$ of length $d+1$ over a $d$-letter alphabet. Comparing to the original combinatorial method, 
the upgrading does not provide a shorter proof; 
however,
it involves of a systematic proof {\it by exhaustion} that 
only consists of checking all the types.
The proof of Theorem \ref{thm:main} is then
achieved with {\it Computer Assistance}.

In practice,  the number of types is very large.
For example, as described in \S2.3 for the study of $F_{11}(x)$, there are
2274558 types!
Fortunately, the set of permutations of each type can be decomposed into the Cartesian product of so-called {\it atoms} (see Substep $3(d)$ in the sequel), and
moreover, the cardinality of each atom can be rapidly evaluated by a sequence of tests (see Definition \ref{def:mu} and Table \ref{tab:Psi}).

\begin{problem}
Is the following power series Apwenian: 
\begin{align*}
	F_{19}(x)&=\Phi(1-x-{x}^{2}-{x}^{3}+{x}^{4}-{x}^{5}+{x}^{6}-{x}^{7}-{x}^{8}+{x}^{9}\\
	&\qquad +{x}^{10}-{x}^{11}-{x}^{12}-{x}^{13}-{x}^{14}-{x}^{15}+{x}^{16}-{x}^{17}- {x}^{18})\,?
\end{align*}
Find a fast computer assisted proof for Theorem \ref{thm:main} 
to answer the above question.
\end{problem}

For proving that $F_{17a}(x)$ is Apwenian, our $C$ program has taken 
about one week by using  24 CPU cores. No hope for $F_{19}(x)$.

\begin{problem}
Find a human proof of Theorem \ref{thm:main} without computer assistance.
\end{problem}

\begin{problem}
	Characterize all the finite $\pm1$-sequences $\bf v$ such that $\Phi(\tilde {\bf v}(x))$ is Apwenian.
\end{problem}

As an application of Theorem \ref{thm:main} in Number Theory, the irrationality exponents of
$F_5(1/b), F_{11}(1/b), F_{17a}(1/b), F_{17b}(1/b)$ are proved to be equal to 2 (see \cite{BHWY2015}).


\section{Proof of Theorem \ref{thm:main}}\label{sec:proof}
Let $d$ be a positive integer and ${\bf v}=(v_0, v_1,v_2,\ldots,v_{d-1})$ be a finite $\pm 1$-sequence of length $d$ with $v_0=1$.
Let ${\bf f}=(f_k)_{k\geq 0}$ be the $\pm 1$-sequence defined by
the following generating function
\begin{equation}\label{def:seq:f}
	\tilde{\bf f}(x) = \Phi(\tilde{\bf v}(x)) = \prod_{k=0}^\infty \tilde{\bf v} (x^{d^k}),
\end{equation}
where $\tilde{\bf v}(x)=\sum_{i=0}^{d-1}v_ix^{i}$. 
The above power series satisfies the following 
functional equation 
\begin{equation}\label{equ:f:func}
\tilde{\bf f}(x)= 
	\tilde{\bf v}(x)\prod_{k=1}^{\infty}\tilde{\bf v}\left(x^{d^{k}}\right)
	=\tilde{\bf v}(x)\tilde{\bf f}(x^{d}).
\end{equation}
The sequence ${\bf f}$  can also be defined by the 
recurrence relations
\begin{equation}\label{equ:f:recurrence}
	f_0=1,\quad f_{dn+i}=v_if_{n} \text{\ for  $n\geq 0$ and $0\leq i \leq d-1$.}
\end{equation}

We divide the set $\{1,2,\ldots,d-1\}$ into two disjoint subsets 
\begin{align*}
	P&=\{1\leq i\leq d-1 \mid v_{i-1} \not= v_i\},\\
	Q&=\{1\leq i\leq d-1 \mid v_{i-1} = v_i\}.
\end{align*}
Two disjoint infinite sets of integers $J$ and $K$ play an important 
in the proof of Theorem \ref{thm:main}.
\begin{defi}\label{def:JK}
If $v_{d-1}=-1$, define
\begin{align*}
	J&=\{(dn+p)d^{2k}-1\ |\ n,k\in N, p\in P\}\\
	 &\qquad\bigcup\{(dn+q)d^{2k+1}-1\ |\ n,k\in N, q\in Q\},\\
	K&=\{(dn+q)d^{2k}-1\ |\ n,k\in N, q\in Q\}\\
	 &\qquad\bigcup\{(dn+p)d^{2k+1}-1\ |\ n,k\in N, p\in P\}.\\
\noalign{If $v_{d-1}=1$, define}
J&=\{(dn+p)d^{k}-1\ |\ n,k\in N, p\in P\},\\
K&=\{(dn+q)d^{k}-1\ |\ n,k\in N, q\in Q\}.
\end{align*}
\end{defi}
From the above definition it is easy to see that $N= J\cup K$.
\begin{lemma}\label{lemma:f=set}
For each $t\ge 0$ the integer $\delta_{t}:=|(f_{t}-f_{t+1})/2|$ is equal to $1$ if and only if $t$ is in $J$.
\end{lemma}
\begin{proof}
Let $t=	(dn+\ell) d^{k}-1$. By \eqref{equ:f:recurrence} we have
\begin{equation*}
f_{(dn+\ell) d^{k}-1} =  f_{d[(dn+\ell)d^{k-1}-1]+(d-1)}
= v_{d-1} f_{(dn+\ell)d^{k-1}-1},
\end{equation*}
and
\begin{equation*}
f_{(dn+\ell) d^{k}-1} = v_{d-1}f_{(dn+\ell)d^{k-1}-1} = \cdots
= v_{d-1}^{k} f_{dn+\ell-1}
= v_{d-1}^{k} v_{\ell-1}f_{n}.
\end{equation*}
In the same manner, 
\begin{equation*}
f_{(dn+\ell) d^{k}} = f_{(dn+\ell)d^{k-1}} = \cdots
= f_{dn+\ell}
= v_{\ell}f_{n}.
\end{equation*}
Hence,
\begin{equation*}
\delta_{t}
=\left|\frac 12 (f_{(dn+\ell) d^{k}-1}-f_{(dn+\ell) d^{k}})\right|
=\left|\frac 12 ( v_{\ell}-  v_{d-1}^kv_{\ell-1})\right|
\end{equation*}
is odd if and only if 
\begin{equation}\label{cond:J}
v_{d-1}^{k}\cdot v_{\ell-1}v_{\ell}=-1.
\end{equation}
There are two cases are to be considered:
(i) If $v_{d-1}=1$, condition \eqref{cond:J} is equivalent to $v_{\ell-1}v_{\ell}=-1$, or $\ell\in P$, or $t\in J$ by Definition \ref{def:JK}.
(ii) If $v_{d-1}=-1$, condition \eqref{cond:J} becomes 
$v_{\ell-1}v_{\ell}=(-1)^{k+1}$, which is equivalent to $\ell \in P$ when $k$ is even and $\ell\in Q$ when $k$ is odd. In other words, $t\in J$. 
\end{proof}
Let $\Sym_{m}=\Sym_{\{0,1,\ldots,m-1\}}$ be the set of all permutations on
	$\{0,1,\ldots,m-1\}$. 
The following Theorem may be viewed as the combinatorial interpretation of Theorem \ref{thm:main}.
\begin{thm}\label{thm:J}
	Let $\bf v$ be a $\pm 1$-sequence of length $d$ with $v_0=1$. The sequence $\bf f$ and the set $J$ associated with $\bf v$ are defined by \eqref{def:seq:f} and Definition \ref{def:JK} respectively.  Then,
the sequence $\bf f$ is Apwenian if, and only if,
the number of permutations $\sigma \in \Sym_{m}$ such that $i+\sigma(i)\in J$ for $i=0,1,\ldots,m-2$ (no constraint on $m-1+\sigma(m-1)\in N$) is an odd integer for every integer $m\ge 1$. 
\end{thm}

\begin{proof}
Let $m$ be a positive integer. By means of elementary transformations the Hankel determinant $H_{m}({\bf f})$ is equal to 
\begin{align*}
H_{m}(\bf f)
	&=
	\left|
	\begin{matrix}
	f_0 		& f_1 	& \cdots & f_{m-1}	\\
	f_1 		& f_2 	& \cdots & f_{m}		\\
	\vdots 	&\vdots & \ddots & \vdots		\\
	f_{m-1}	& f_{m}	& \cdots & f_{2m-2}
	\end{matrix}\right|\\
&=
	\left|\begin{matrix}
	f_0-f_1 		& f_1-f_2 			& \cdots & f_{m-2}-f_{m-1}	& f_{m-1}		\\
	f_1-f_2 		& f_2-f_3 			& \cdots & f_{m-1}-f_{m} 		& f_{m}			\\
	\vdots 			& \vdots 				& \ddots & \vdots						& \vdots		\\
	f_{m-1}-f_m	& f_{m}-f_{m+1}	& \cdots & f_{2m-3}-f_{2m-2}& f_{2m-2}
	\end{matrix}\right| \\
&=  2^{m-1} \times
 	\left|\begin{matrix}
	\frac{f_0-f_1}{2} 		& \frac{f_1-f_2}{2} 			& \cdots & \frac{f_{m-2}-f_{m-1}}{2}	& f_{m-1}		\\
	\frac{f_1-f_2}{2} 		& \frac{f_2-f_3}{2} 			& \cdots & \frac{f_{m-1}-f_{m}}{2} 		& f_{m}			\\
	\vdots 								& \vdots 									& \ddots & \vdots											& \vdots		\\
	\frac{f_{m-1}-f_m}{2}	& \frac{f_{m}-f_{m+1}}{2}	& \cdots & \frac{f_{2m-3}-f_{2m-2}}{2}& f_{2m-2}
	\end{matrix}\right|. 
\end{align*}
By Lemma \ref{lemma:f=set}, we have
\begin{equation}\label{eqn:det}
\frac{H_{m}(\bf f)}{2^{m-1}}\equiv
\left|\begin{matrix}
\delta_0 		& \delta_1 	& \cdots & \delta_{m-2} 	& 1	  		\\
\delta_1 		& \delta_2 	& \cdots & \delta_{m-1} 	& 1				\\
\vdots 	&\vdots & \ddots & \vdots 	& \vdots	\\
\delta_{m-1}	& \delta_{m}	& \cdots & \delta_{2m-3} & 1
\end{matrix}\right|
\pmod2.
\end{equation}
By the very definition of a determinant or the Leibniz formula, the determinant
occurring on the right-hand side of the congruence \eqref{eqn:det}
is equal to
\begin{equation}\label{eqn:sumt}
	S_m:=\sum_{\sigma\in \Sym_m}(-1)^{\inv(\sigma)}\delta_{0+\sigma_0}\delta_{1+\sigma_1}\cdots \delta_{m-2+\sigma_{m-2}},
\end{equation}
where $\inv(\sigma)$ is the number of inversions of the permutation $\sigma$.
By Lemma \ref{lemma:f=set} the product $\delta_{0+\sigma_0}\delta_{1+\sigma_1}\cdots \delta_{m-2+\sigma_{m-2}}$ is equal to 1 if $i+\sigma_i\in J$ for $i=0,1,\ldots,m-2$, and to 0 otherwise. Hence, the summation $S_m$ 
is congruent modulo 2 to the number of permutations $\sigma\in\Sym_{m}$ such that $i+\sigma_{i}\in J$ for all $i=0,1,\ldots,m-2$. Hence, $\bf f$ is Apwenian if and only if the number of permutations $\sigma \in \Sym_{m}$ such that $i+\sigma(i)\in J$ for $i=0,1,\ldots,m-2$ (no constraint on $m-1+\sigma(m-1)\in N$) is an odd integer for every integer $m\ge 1$. 
\
\end{proof}

For proving that the sequence $\bf f$ is Apwenian by means of Theorem~\ref{thm:J}, it is convenient to introduce the following notations.
\begin{defi}\label{def:jk}
For $m\ge \ell \ge 0$ let
$\mathfrak{J}_{m,\ell}$ (resp. $\mathfrak{K}_{m,\ell}$) be the set of all permutations $\sigma =\sigma_0 \sigma_1 \cdots \sigma_{m-1} \in \Sym_{m}$ such that $i+\sigma_i\in J$ (resp. $i+\sigma_i\in K$) for $i\in \{0,1,\ldots,m-1\}\setminus \{\ell\}$.
Let $n\geq 1$; 
for simplicity,  write:
\begin{align*}
		j_{m,\ell}&:=\#\mathfrak{J}_{m,\ell},
		 &k_{m,\ell}&:=\#\mathfrak{K}_{m,\ell},\\
		X_{n}	& := \sum_{i=0}^{n-1}j_{n,i}, &Y_{n}&:=j_{n,n},&&Z_{n}:=j_{n,n-1},\\
		U_{n}	& := \sum_{i=0}^{n-1}k_{n,i}, &V_{n}&:=k_{n,n},&&W_{n}:=k_{n,n-1},\\
	T_n	& :=X_n+X_nY_n+Y_n, &&& \\
	R_n	& :=U_n+U_nV_n+V_n. &&&
\end{align*}
\end{defi}
Notice that if $\ell=m$, then $\{0,1,\ldots, m-1\}\setminus\{\ell\}=\{0,1,\ldots, m-1\}$, so that
$j_{m,m}$ (resp. $k_{m,m}$) is the
number of permutations $\sigma\in \Sym_{m}$ such that 
$i+\sigma(i)\in J$ (resp. $\in K$) for all $i$.

\medskip

By Theorem \ref{thm:J} and Definition \ref{def:jk} the sequence $\bf f$ is Apwenian if and only if $Z_n\equiv 1 \pmod 2$. 
In Section \ref{sec:Algorithm} we describe an algorithm enabling us to {\it find} and also {\it prove} 
a list of recurrence relations between $X_n, Y_n, X_n, U_n, V_n, W_n$. 
Then, it is routine to check whether $Z_n\equiv 1\pmod 2$ or not.
Our program {\tt Apwen.py} is an implementation of the latter algorithm in Python.

We now produce the proof of Theorem \ref{thm:main} by means of the program {\tt Apwen.py}.  Since $F_2(x)$ has been proved to be Apwenian in \cite{APWW1998},  only the
three power series $F_3(x), F_5(x)$ and $F_{11}(x)$ 
require our attention.
We can also prove that $F_{13}(x), F_{17a}(x), F_{17b}(x)$ are Apwenian in the same manner. However, the full proofs are lengthy and are not reproduced in the paper. 

\goodbreak
\subsection{$F_3(x)$ is Apwenian} 
Take ${\bf v}=(1, -1, -1)$
with $d=3$ and $v_{d-1}=-1$. Then,  the corresponding infinite $\pm1$-sequence ${\bf f}$ is equal to $F_{3}(x)$. We have
$P=\{1\}$, $Q=\{2\}$ and
\begin{align*}
	J &
		=\{(3n+1)3^{2k}-1\mid n,k\in N\}\cup\{(3n+2)3^{2k+1}-1\mid n,k\in N\}				\\
	  &
		=\{0,3,5,6,8,9,12,14,15,18,\ldots\},	\\
	K &
		=\{(3n+2)3^{2k}-1 \mid n,k\in N\}\cup\{(3n+1)3^{2k+1}-1\mid n,k\in N\}				\\
	  &
		=\{1,2,4,7,10,11,13,16,17,\ldots\}=N\setminus J.
\end{align*}
By enumerating a list of 24 {\it types} of permutations (see Section 
\ref{sec:CountType}), the program {\tt Apwen.py} finds and proves the following recurrences.
 \begin{lemma}\label{lemma:main3XYZ}
For each $n\ge 1$ we have
\begin{align*}
	X_{3n+0}	&\equiv U_n, &                        	Y_{3n+0}	&\equiv U_n+V_n								,\\
	X_{3n+1}	&\equiv W_{n+1}(U_n+V_n), &           	Y_{3n+1}	&\equiv W_{n+1} V_{n}					,\\
	X_{3n+2}	&\equiv W_{n+1}(U_{n+1}+V_{n+1}), &   	Y_{3n+2}	&\equiv W_{n+1} V_{n+1}				,\\
	\noalign{\medskip}
	Z_{3n+0}	&\equiv W_{n} (U_n+U_nV_n+V_n)	,\\
	Z_{3n+1}	&\equiv W_{n+1} (U_n+U_nV_n+V_n),\\
	Z_{3n+2}	&\equiv W_{n+1}							.	
\end{align*}
\end{lemma}
As explained in Section \ref{sec:CountType}, the above relations express $X,Y,Z$ in function of $U,V,W$ since $v_{d-1}=-1$. By exchanging the values of $P$ and $Q$, $J$ and $K$, the program {\tt Apwen.py} yields other relations
which express $U,V,W$ in terms of $X,Y,Z$ by enumerating a list of 26 types of permutations.
\begin{lemma}\label{lemma:main3UVW}
For each $n\ge 1$ we have
\begin{align*}
	U_{3n+0}	&\equiv X_n										,	&            	V_{3n+0}	&\equiv X_n+Y_n  							,\\
	U_{3n+1}	&\equiv Z_{n+1} Y_n						,&	           	V_{3n+1}	&\equiv Z_{n+1} X_{n}  				,\\
	U_{3n+2}	&\equiv Z_{n+1} Y_{n+1}				,&	           	V_{3n+2}	&\equiv Z_{n+1} X_{n+1}				,\\
	\noalign{\medskip}
	W_{3n+0}	&\equiv Z_{n} (X_n+X_nY_n+Y_n)	,\\
	W_{3n+1}	&\equiv Z_{n+1} (X_n+X_nY_n+Y_n),\\
	W_{3n+2}	&\equiv Z_{n+1}								.\\
\end{align*}
\end{lemma}
From Lemmas \ref{lemma:main3XYZ} and \ref{lemma:main3UVW}
we obtain the following ``simplified'' recurrence relations based on some elementary calculations.
\begin{coroll}\label{cor:moresimp3} For each positive integer $n$ we have
  \begin{align*}
	Z_{3n+0}	&\equiv W_n R_n					, & W_{3n+0}	&\equiv Z_n T_n					,\\
	Z_{3n+1}	&\equiv W_{n+1} R_n			, & W_{3n+1}	&\equiv Z_{n+1} T_n			,\\
	Z_{3n+2}	&\equiv W_{n+1}					, & W_{3n+2}	&\equiv Z_{n+1}								,\\
	T_{3n+0}	&\equiv R_n							, & R_{3n+0}	&\equiv T_n										,\\
	T_{3n+1}	&\equiv W_{n+1} R_n			, & R_{3n+1}	&\equiv Z_{n+1} T_n			,\\
	T_{3n+2}	&\equiv W_{n+1} R_{n+1}	, & R_{3n+2}	&\equiv Z_{n+1} T_{n+1}.
   \end{align*}
  \end{coroll}
	Since $Z_1=1, T_1=3, W_1=1, R_1=1, Z_2=1, T_2=1, W_2=1$ and $R_2=7$, Corollary \ref{cor:moresimp3} 
	yields 
	$Z_{m}\equiv T_m \equiv W_m \equiv R_m \equiv 1 \pmod 2$ for every positive integer $m$
by induction.  Hence, $F_{3}(x)$ is Apwenian.

\subsection{$F_5(x)$ is Apwenian}
Take ${\bf v}=(1, -1, -1, -1, 1)$ with $d=5$ and $v_{d-1}=1$. Then,  the corresponding infinite $\pm1$-sequence ${\bf f}$ is equal to $F_5(x)$. We have
\begin{align*}
	P&=\{1,4\}, \\
	Q&=\{2,3\}, \\
J&=\{(5n+1)5^{k}-1\ |\ n,k\in N\}\cup \{(5n+4)5^{k}-1\ |\ n,k\in N\}\\
	&=\{0, 3, 4, 5, 8, 10, 13, 15, 18, 19, 20, 23, 24, 25, 28, 29, 30, 33,  \ldots\},\\
K&=\{(5n+2)5^{k}-1\ |\ n,k\in N\}\cup \{(5n+3)5^{k}-1\ |\ n,k\in N\}\\
	&=\{1, 2, 6, 7, 9, 11, 12, 14, 16, 17, 21, 22, 26, 27, 31, 32, 34, 36, \ldots\}.
\end{align*}
By enumerating a list of 225 {\it types} of permutations, the Python program {\tt Apwen.py} finds and proves the following recurrences.
\begin{lemma}\label{lemma:main5}
	For each $n\geq 1$ we have
\begin{align*}
X_{5n+0}&\equiv X_{n},  &                         Y_{5n+0}&\equiv Y_n,\\
X_{5n+1}&\equiv Z_{n+1} Y_{n},  &                 Y_{5n+1}&\equiv Z_{n+1} (X_{n}+Y_{n}),\\
X_{5n+2}&\equiv Z_{n+1} (X_{n}+Y_{n}),  &         Y_{5n+2}&\equiv Z_{n+1} X_{n},\\
X_{5n+3}&\equiv Z_{n+1} (X_{n+1}+Y_{n+1}),  &     Y_{5n+3}&\equiv Z_{n+1} X_{n+1},\\
X_{5n+4}&\equiv Z_{n+1} Y_{n+1},  &               Y_{5n+4}&\equiv Z_{n+1} (X_{n+1}+Y_{n+1}),\\
	\noalign{\medskip}
Z_{5n+0}&\equiv \rlap{$Z_{n}(X_n+X_nY_n+Y_n),$} &\\
Z_{5n+1}&\equiv \rlap{$Z_{n+1}(X_n+X_nY_n+Y_n),$} &\\
Z_{5n+2}&\equiv \rlap{$Z_{n+1}(X_n+X_nY_n+Y_n),$} &\\
Z_{5n+3}&\equiv \rlap{$Z_{n+1},$} &\\
Z_{5n+4}&\equiv \rlap{$Z_{n+1}(X_{n+1}+X_{n+1}Y_{n+1}+Y_{n+1}).$} &
\end{align*}
\end{lemma}
As explained in Section \ref{sec:Algorithm}, the above relations are between $X,Y,Z$ without involving $U,V,W$, since $v_{d-1}=1$.
We obtain the following ``simplified'' recurrence relations based on some elementary calculations.
  \begin{coroll}\label{cor:moresimp5} For each positive integer $n$ we have
\begin{align*}                              
	Z_{5n+0}&\equiv Z_{n} T_n, &        T_{5n+0}&\equiv T_n,\\        
	Z_{5n+1}&\equiv Z_{n+1} T_n,&       T_{5n+1}&\equiv Z_{n+1} T_n,\\
	Z_{5n+2}&\equiv Z_{n+1} T_n,&       T_{5n+2}&\equiv Z_{n+1} T_n,\\
	Z_{5n+3}&\equiv Z_{n+1}, &        	T_{5n+3}&\equiv Z_{n+1} T_{n+1},\\
	Z_{5n+4}&\equiv Z_{n+1} T_{n+1},&   T_{5n+4}&\equiv Z_{n+1} T_{n+1}.
\end{align*}
\end{coroll}

Since $Z_1=1,T_1=3,Z_2=1,T_2=1,Z_3=1,T_3=9,Z_4=5,T_4=129$, Corollary \ref{cor:moresimp5} yields 
$Z_{n}\equiv T_n  \equiv 1 \pmod 2$ for each $n\ge 1$ by induction.
Hence, $F_{5}(x)$ is Apwenian.

\subsection{$F_{11}(x)$ is Apwenian}
Take 
$${\bf v}=(1, -1, -1, 1, -1, 1, 1, 1, 1, -1, -1)$$ 
with $d=11$ and $v_{d-1}=-1$. Then,  the corresponding infinite $\pm1$-sequence ${\bf f}$ is equal to $F_{11}(x)$. We have
\begin{align*}
	P&=\{1, 3, 4, 5, 9\}, \\
	Q&=\{2, 6, 7, 8, 10\}, \\
J& =\{0, 2, 3, 4, 8, 11, 13, 14, 15, 19, 21, 22, 24, 25, 26, 30, 33, 35,  \ldots\},\\
K&= \{1, 5, 6, 7, 9, 10, 12, 16, 17, 18, 20, 23, 27, 28, 29, 31, 32, 34, \ldots\}.
\end{align*}
By enumerating a list of 2274558 {\it types} of permutations, the program {\tt Apwen.py} finds and proves the following recurrences.
\begin{lemma}\label{lemma:main11XYZ}
For each $n\geq 1$ we have
\begin{align*}
    X_{11n+0}&\equiv U_{n} , &                               Y_{11n+0}&\equiv U_{n} + V_{n}, \\
		X_{11n+1}&\equiv W_{n+1}(V_{n} + U_{n}), &               Y_{11n+1}&\equiv V_{n}W_{n+1}, \\
    X_{11n+2}&\equiv U_{n}W_{n+1}, &                     	Y_{11n+2}&\equiv W_{n+1} (V_{n}+ U_{n}), \\
	X_{11n+3}&\equiv W_{n+1}(V_{n} + U_{n}), &                 Y_{11n+3}&\equiv V_{n}W_{n+1}, \\
    X_{11n+4}&\equiv V_{n}W_{n+1}, &                         Y_{11n+4}&\equiv U_{n}W_{n+1}, \\
    X_{11n+5}&\equiv U_{n}W_{n+1}, &                     	Y_{11n+5}&\equiv W_{n+1}(V_{n}+ U_{n}), \\
    X_{11n+6}&\equiv U_{n+1}W_{n+1}, &                   	Y_{11n+6}&\equiv W_{n+1}(U_{n+1} + V_{n+1}), \\
    X_{11n+7}&\equiv V_{n+1}W_{n+1}, &                       Y_{11n+7}&\equiv U_{n+1}W_{n+1}, \\
	X_{11n+8}&\equiv W_{n+1}(U_{n+1}+ V_{n+1}), &              Y_{11n+8}&\equiv V_{n+1}W_{n+1}, \\
    X_{11n+9}&\equiv U_{n+1}W_{n+1}, &                   	Y_{11n+9}&\equiv W_{n+1}(U_{n+1}+ V_{n+1}), \\
	X_{11n+10}&\equiv W_{n+1}(U_{n+1}+ V_{n+1}), &             Y_{11n+10}&\equiv V_{n+1}W_{n+1}, \\ 
	\noalign{\medskip}
	Z_{11n+0}&\equiv \rlap{$W_{n}(V_{n} + U_{n} + U_{n}V_{n}),$} &\\
	Z_{11n+i}&\equiv \rlap{$W_{n+1}(U_{n}V_{n} + V_{n} + U_{n}), \qquad (i=1,2,3,4,5)$} \\
	Z_{11n+6}&\equiv \rlap{$W_{n+1},$} \\
	Z_{11n+i}&\equiv \rlap{$W_{n+1}(U_{n+1} + U_{n+1}V_{n+1} + V_{n+1}). \qquad (i=7,8,9,10)$} 
\end{align*}
\end{lemma}
As explained in Section \ref{sec:Algorithm}, the above relations express $X,Y,Z$ in function of $U,V,W$, since $v_{d-1}=-1$. By exchanging the values of $P$ and $Q$, $J$ and $K$, the program {\tt Apwen.py} yields another list of relations
which express $U,V,W$ in terms of $X,Y,Z$. For this purpose, a long list of 2350964 {\it types} of permutations are enumerated. 
\goodbreak
\begin{lemma}\label{lemma:main11UVW}
For each $n\geq 1$ we have
\begin{align*}
  U_{11n+0}&\equiv X_{n}, &                          V_{11n+0}&\equiv X_{n} + Y_{n}, \\
    U_{11n+1}&\equiv Y_{n}Z_{n+1}, &                   V_{11n+1}&\equiv X_{n}Z_{n+1}, \\
    U_{11n+2}&\equiv X_{n}Z_{n+1}, &               	V_{11n+2}&\equiv Z_{n+1}(X_{n}+ Y_{n}), \\
    U_{11n+3}&\equiv Y_{n}Z_{n+1}, &                   V_{11n+3}&\equiv X_{n}Z_{n+1}, \\
	U_{11n+4}&\equiv Z_{n+1}(X_{n}+ Y_{n}), &            V_{11n+4}&\equiv Y_{n}Z_{n+1}, \\
    U_{11n+5}&\equiv X_{n}Z_{n+1}, &               	V_{11n+5}&\equiv Z_{n+1}(X_{n}+ Y_{n}), \\
    U_{11n+6}&\equiv X_{n+1}Z_{n+1}, &             	V_{11n+6}&\equiv Z_{n+1}(Y_{n+1}+ X_{n+1}), \\
	U_{11n+7}&\equiv Z_{n+1}(Y_{n+1}+ X_{n+1}), &        V_{11n+7}&\equiv Y_{n+1}Z_{n+1}, \\
    U_{11n+8}&\equiv Y_{n+1}Z_{n+1}, &                 V_{11n+8}&\equiv X_{n+1}Z_{n+1}, \\
    U_{11n+9}&\equiv X_{n+1}Z_{n+1}, &             	V_{11n+9}&\equiv Z_{n+1}(Y_{n+1}+ X_{n+1}), \\
    U_{11n+10}&\equiv Y_{n+1}Z_{n+1}, &                V_{11n+10}&\equiv X_{n+1}Z_{n+1}, \\
	\noalign{\medskip}
	W_{11n+0}&\equiv \rlap{$Z_{n}(Y_{n} + X_{n} + X_{n}Y_{n}),$} \\
	W_{11n+i}&\equiv \rlap{$Z_{n+1}(X_{n}Y_{n} + X_{n} + Y_{n}), \qquad (i=1,2,3,4,5)$} \\
    W_{11n+6}&\equiv \rlap{$Z_{n+1},$} \\
	W_{11n+i}&\equiv \rlap{$Z_{n+1}(Y_{n+1} + X_{n+1}Y_{n+1} + X_{n+1}). \qquad (i=7,8,9,10)$} 
\end{align*}
\end{lemma}
From Lemmas \ref{lemma:main11XYZ} and \ref{lemma:main11UVW}
we obtain the following ``simplified'' recurrence relations based on some elementary calculations.
\begin{coroll}\label{cor:moresimp11} For each positive integer $n$ we have
\begin{align*}
T_{11n+0}&\equiv R_n, &                     R_{11n+0}&\equiv T_{n}, \\
	T_{11n+i}&\equiv R_n  W_{n+1}, &            R_{11n+i}&\equiv T_{n}Z_{n+1},&\quad (1\leq i\leq 5) \\
	T_{11n+i}&\equiv R_{n+1}  W_{n+1}, &        R_{11n+i}&\equiv T_{n+1}Z_{n+1},&\quad (6\leq i \leq 10) \\
Z_{11n+0}&\equiv R_{n}W_{n}, &              W_{11n+0}&\equiv T_nZ_n, \\
	Z_{11n+i}&\equiv R_{n}W_{n+1}, &            W_{11n+i}&\equiv T_nZ_{n+1}, &\quad (1\leq i \leq 5) \\
Z_{11n+6}&\equiv W_{n+1}, &                 W_{11n+6}&\equiv Z_{n+1}, \\
	Z_{11n+i}&\equiv R_{n+1}W_{n+1}, &          W_{11n+i}&\equiv T_{n+1}Z_{n+1}. &\quad (7\leq i \leq 10) 
\end{align*}
\end{coroll}
The first values of $Z_m, T_m, W_m, R_m$ are reproduced in the following table.
{\small
\begin{equation*}
\begin{array}{|c|rrrrrrrrrr| }
\hline
m& 1 & 2 & 3 & 4 & 5 & 6 & 7 & 8 & 9 & 10\\
\hline
Z_m& 1 & 1 & 3 & 11 & 13 & 25 & 39 & 117 & 739 &4431\\
T_m& 3 & 5 & 47&237&487&419&3503 & 66905 & 3527039&82080975\\
W_m& 1 & 1 & 1 & 1 & 5 & 25 & 177 & 1091 & 3839& 19791\\
R_m& 1 & 5 & 1 & 11 & 107 &\!\! 5151 &\!\! 198769 &\!\! 4802755 &\!\! 56576127&\!\! 2717644635\\
\hline
\end{array}
\end{equation*}
}	

Corollary \ref{cor:moresimp11} yields 
	$Z_{m}\equiv T_m \equiv W_m \equiv R_m \equiv 1 \pmod 2$ for every positive integer $m$
by induction.  Hence, $F_{11}(x)$ is Apwenian.

\section{Algorithm for finding the recurrences}\label{sec:CountType}
Keep the same {\it notations} as in Section \ref{sec:proof}.
We will show how  to {\it find} and also {\it prove} 
a list of recurrence relations between the quantities $X_n, Y_n, Z_n, U_n, V_n, W_n$. 
The set $N$ of nonnegative integers is partitioned into $d$ disjoint subsets 
$A_0, A_1, \ldots, A_{d-1}$
according to the value  modulo $d$:
\begin{equation}\label{equ:setA}
		A_{i}=\{dn+i\ |\ n\in N\}\quad (i=0,1,\ldots,d-1).
\end{equation}
For an infinite set $S$ let $S|_{m}$ be the set composed of the $m$ smallest integers in $S$.
Let $\beta:N\rightarrow N$ denote the transformation 
$k\mapsto \lfloor \frac kd \rfloor$. In other words,
\begin{equation}\label{equ:beta}
	\beta(k)=(k-i)/d \qquad\text{if $k\in A_{i}$.} 
\end{equation}
For simplicity, write
$$
\bar{J}=
\begin{cases}
	J & \text{if\ } v_{d-1}=1,\\
	K & \text{if\ } v_{d-1}=-1,
\end{cases}
\quad\text{and}\quad
\bar{K}=
\begin{cases}
	K & \text{if\ } v_{d-1}=1,\\
	J & \text{if\ } v_{d-1}=-1.
\end{cases}
$$
Then $\bar{\SymJ}_{m,\ell},\bar{X}_{n},\bar{Y}_{n},\bar{Z}_{n}$ mean $\SymJ_{m,\ell},X_{n},Y_{n},Z_{n}$ (resp. $\SymK_{m,\ell},U_{n},V_{n},W_{n}$) if $v_{d-1}=1$ (resp. $v_{d-1}=-1$).
\begin{lemma}\label{lemma:set}
For each $p\in P$ and $q\in Q$ we have
\begin{itemize}
\item[(i)] $A_{p-1}\subset J$ and $A_{q-1}\subset K$;
\item[(ii)]$A_{q-1}\cap J=\emptyset$ and $A_{p-1}\cap K=\emptyset$;
\item[(iii)]$\beta(A_{d-1}\cap J)= \bar{J}$ and $\beta(A_{d-1}\cap K)= \bar{K}$. 
\end{itemize}
\end{lemma}
\begin{proof}
(i) By the definition of the sets $J$ and $K$ we have
\begin{align*}
	J&\supset \{(dn+p) d^{2k} -1 \mid n,k\in N\} \supset \{(dn+p) d^{0} -1 \mid n\in N\}=A_{p-1};\\
	K&\supset \{(dn+q) d^{2k} -1 \mid n,k\in N\} \supset \{(dn+q) d^{0} -1 \mid n\in N\}=A_{q-1}.
\end{align*}

(ii) By (i) and the relation $K\cap J=\emptyset$.

(iii) By Definition \ref{def:JK} with $v_{d-1}=-1$, we have
\begin{align*}
	A_{d-1}\cap J&=\{(dn+p)d^{2k+2}-1\ |\ n,k\in N, p\in P\}\\
	 &\qquad\bigcup\{(dn+q)d^{2k+1}-1\ |\ n,k\in N, q\in Q\}.\\
	\noalign{Thus,}
	\beta(A_{d-1}\cap J)&=\{(dn+p)d^{2k+1}-1\ |\ n,k\in N, p\in P\}\\
	 &\qquad\bigcup\{(dn+q)d^{2k}-1\ |\ n,k\in N, q\in Q\}.\\
	 &=K=\bar{J}.
\end{align*}

 If $v_{d-1}=1$, then
$A_{d-1}\cap J=\{(dn+p)d^{k+1}-1\ |\ n,k\in N, p\in P\}$.
Thus,
$\beta(A_{d-1}\cap J)=\{(dn+p)d^{k}-1\ |\ n,k\in N, p\in P\}=J=\bar{J}$.
The second part $\beta(A_{d-1}\cap K)=\bar{K}$ is proved in the same manner.
\end{proof}

Let $0\leq i,j\leq d-1$ and $x\in A_i, y\in A_j$. 
For determining the condition of $i$ and $j$ such that the sum $x+y$ belongs to  $J$ or $K$, there are three cases to be considered.

(S1) If $i+j+1 \pmod d \in P$, then, $x+y\in J$;

(S2) If $i+j+1 \pmod d \in Q$, then, $x+y\in K$;

(S3) If $i+j+1 \pmod d=0$, then, $x+y\in A_{d-1}$. In this case, the sum $x+y$ may belong to $J$ or $K$.

\medskip

Let $m\geq \ell \geq 0$.
We want to enumerate the permutations in $\SymJ_{m,\ell}$ modulo~2. 
Each permutation
$\sigma=\sigma_{0}\sigma_{1}\cdots \sigma_{m-1}\in\Sym_m$ may be written as the
{\it two-line representation}
$$
	\begin{pmatrix}
	 0  		&  1 &  2 		&  \cdots  		  		&  m-1  		\\
		\sigma_{0} &  \sigma_{1} &\sigma_{2}			& \cdots 	 		& \sigma_{m-1} 
	\end{pmatrix}.
$$
The columns $\binom{i}{\sigma_{i}}$ are called {\it biletters}.
For each $\sigma\in \SymJ_{m,\ell}$ a biletter $\binom{i}{\sigma_{i}}$ in $\sigma$ 
is said to be of {\it $($normal$)$ form} $\binom{a_j}{a_k}$ 
(resp. {\it specific form} $\binom{\ell}{a_k}$ )
if $i\neq \ell$ and
$(i, \sigma_{i})\in A_j\times A_k$ 
(resp. $i=\ell$ and $\sigma_{i}\in A_{k}$).
To count the permutations from $\SymJ_{m,\ell}$ modulo 2, we proceed in several 
steps.  In most cases the calculations are illustrated with $d=5$.

\medskip

{\it Step 1. Occurrences of biletters.} Since we want to enumerate permutations modulo 2, we can delete suitable pairs of the permutations and the result will not be changed. Let $i\in N|_d$, if a permutation $\sigma\in \SymJ_{m,\ell}$ contains more than two biletters of form 
$\binom{a_i}{a_j}$ such that $i+j+1\pmod d \in P$,
select the first two such biletters $\binom{i_1}{j_1}$ and $\binom{i_2}{j_2}$. We define another permutation $\tau$ obtained from $\sigma$ by exchanging $j_1$ and $j_2$ in the bottom line. This procedure is reversible. 
By (S1), 
it is easy to verify that $\tau$ is also in $\mathfrak{J}_{m,\ell}$, so that we can delete the pair $\sigma$ and $\tau$. Then, there only remain the permutations containing 0 or 1 biletter of form $\binom{a_i}{a_j}$  such that 
$i+j+1\pmod d \in P$.

Let $\mathfrak{J}'_{m,\ell}$ be the set of permutations 
$\sigma\in \mathfrak{J}_{m,\ell}$ which, 
for each $i\in N|_d$,
contains 0 or 1 biletter of form $\binom{a_i}{a_j}$ such that $i+j+1\pmod d \in P$. 
We have $j_{m,\ell}=\#\mathfrak{J}_{m,\ell}\equiv \#\mathfrak{J}'_{m,\ell} \pmod 2$.
By (S2), each permutation $\sigma\in\mathfrak{J}'_{m,\ell}$ does not contain any biletter of form $\binom{a_i}{a_j}$ such that $i+j+1\pmod d \in Q$. 
Thus, most of the biletters are of form $\binom{a_i}{a_{d-i-1}}$. 
In conclusion, the number of occurrences of each form is summarized in 
Table 3.1.
A biletter of form
$\binom{a_i}{a_j}$ such that $i+j+1\pmod d \in P$
is said to be {\it unsociable}.
A biletter of form
$\binom{a_j}{a_{d-j-1}}$ 
is said to be {\it friendly}.
By Table~3.1, each permutation in $\SymJ_{m,\ell}'$ contains only a few 
unsociable biletters.
\begin{table}[!ht]\label{tab:Occur}
$$
\begin{array}{ l | l}
	\hfil\text{form}   &  \hfil\text{total times} \\
	\hline 
	\\ [-1.1em]
	\{\binom{a_k}{a_j} \mid k+j+1\pmod d \in Q\} & 0 \\[0.4em]
	\{\binom{a_i}{a_j} \mid i+j+1\pmod d \in P\} \text{ for each } i\in N|_d & 0,  1 \\[0.4em]
	\{\binom{a_j}{a_{d-j-1}} \mid j\in N|_d\} & 0, 1, 2, 3, \ldots\\[0.3em]
  \hline \\[-1.1em]
	\{\binom{\ell}{a_j} \mid j \in N|_d\}\quad (\ell=m)& 0\\[0.3em]
	\{\binom{\ell}{a_j} \mid j \in N|_d\}\quad (0\leq \ell\leq m-1)& 1\\[0.3em]
	\hline
\end{array}
$$
\caption{Number of occurrences of biletters}
\end{table}
{\it Step 2. Form and type.} The two-line representation of a permutation can be seen as a word of biletters. In fact, the order of the biletters does not matter. 
Let $m\geq 2d$.
The {\it form} $f(\sigma)$ of a permutation $\sigma\in\mathfrak{J}'_{m,\ell}$  is obtained from $\sigma$ by 
replacing each biletter of $\sigma$ by its (normal or specific) form.
From Table 3.1, the form $f(\sigma)$ of a permutation 
$\sigma\in \SymJ'_{m,\ell}$ is
\begin{equation}\label{equ:s0s1s2}
\left(
\begin{array}{c  |c |c|c }
	\ccch{a_0}  \ccch{a_0}  \cch{a_0} & \ccch{a_1}  \ccch{a_1}  \cch{a_1} &\ldots & \ccch{a_{d-1}}  \ccch{a_{d-1}}  \ccch{a_{d-1}}\\
	\ccch{a_{d-1}}  \ccch{a_{d-1}}  \cch{s_0} & \ccch{a_{d-2}}  \ccch{a_{d-2}} \cch{s_1} &\ldots & \ccch{a_{0}} \ccch{a_0} \ccch{s_{d-1}}
\end{array}
\right),
\end{equation}
for $\ell=m$, 
or
\begin{equation}\label{equ:s0s1s2l}
\left(
\begin{array}{c  |c |c|c |c}
	\ccch{a_0}  \ccch{a_0}  \cch{a_0} & \ccch{a_1}  \ccch{a_1}  \cch{a_1} &\ldots & \ccch{a_{d-1}}  \ccch{a_{d-1}}  \ccch{a_{d-1}}&\ell\\
	\ccch{a_{d-1}}  \ccch{a_{d-1}}  \cch{s_0} & \ccch{a_{d-2}}  \ccch{a_{d-2}} \cch{s_1} &\ldots & \ccch{a_{0}} \ccch{a_0} \ccch{s_{d-1}}&s_{d}
\end{array}
\right),
\end{equation}
for $0\leq \ell\leq  m-1$,
where 
\begin{equation}\label{cond:type}
\begin{cases}
s_i &\in \{a_j\mid i+j+1\pmod d\in P\cup\{0\}\}, \quad (i\in N|_d)  \\
s_{d}&\in \{a_0,a_1,\ldots,a_{d-1}\}.
\end{cases}
\end{equation}
Consequently, it
can be characterized by 
a word $t(\sigma)=s_0s_1\ldots s_{d-1}$ 
or $s_0s_1\ldots s_{d-1}s_{d}$,
of length $d$ or $d+1$ respectively.
The word $t(\sigma)$ is called the {\it type} of the permutation $\sigma$.
We classify the permutations from the set $\SymJ'_{m,\ell}$ according to 
the type $t=s_0s_1\ldots s_{d-1}$ (resp. $t=s_0s_1\ldots s_{d-1}s_d$) by defining
\begin{equation}\label{def:Jtype}
\SymJ_{m,\ell}^{t}=\{\sigma\in \SymJ'_{m,\ell}\ |\ t(\sigma)=t\}.
\end{equation}
Hence,
\begin{equation}\label{equ:sumJtype}
j_{m,\ell} \equiv \sum_{t} \#\mathfrak{J}^t_{m,\ell} \pmod2.
\end{equation}

Some types do not have any contribution for counting the permutations modulo 2, as stated in the following two lemmas.

\begin{lemma}\label{lemma:eq0}
	Let $\ell=m$ and $t=s_0s_1s_2\ldots s_{d-1}$ $($resp.
	$\ell\in N|_{m}$ and $t=s_0s_1s_2\ldots s_{d-1}s_d$ $)$. 
	If there are $0\le i<j\le d-1$ $($resp. $0\le i<j\le d$ $)$, 
	such that $s_i=s_j$, $s_i\neq a_{d-i-1}$ and $s_j\neq a_{d-j-1}$, then 
\begin{equation}\label{equ:eq0}
\#\SymJ_{m,\ell}^{t}\equiv 0 \pmod2.
\end{equation}
\end{lemma}
\begin{proof}
If $\SymJ_{m,\ell}^{t}=\emptyset$, then \eqref{equ:eq0} holds. 
Otherwise, each permutation $\sigma\in \SymJ_{m,\ell}^t$
has two biletters $\binom{i_1}{i_2}$ and $\binom{j_1}{j_2}$ of forms $\binom{a_i}{a_k}$ and $\binom{a_j}{a_k}$, respectively, where $a_k=s_i=s_j$. 
We define another permutation $\tau$ obtains from $\sigma$ by exchanging $i_2$ and $j_2$ in the bottom line. 
This procedure is reversible. By Lemma \ref{lemma:set}(i) or Table 3.1, it is easy to verify that $\tau$ is also in $\SymJ_{m,\ell}^{t}$. 
Thus, the transformation $\sigma\leftrightarrow \tau$ is an {\it involution} on $\mathfrak{J}_{m,\ell}^{t}$. 
Hence, $\#\SymJ_{m,\ell}^{t}\equiv 0\pmod2$.
\end{proof}
\begin{lemma}\label{lemma:eqsum0}
Let $\ell\in N|_m$ and $t=s_0s_1s_2\ldots s_d$. If there is $i\in N|_m$ such that $s_i\neq a_{d-i-1}$, then
\begin{equation}\label{equ:eqsum0}
\sum_{\ell\in N|_{m}\cap A_i}\#\SymJ_{m,\ell}^{t}\equiv 0 \pmod2.
\end{equation}
\end{lemma}
\begin{proof}
For any $\ell\in N|_m\cap A_i$, each permutation  
$\sigma\in\SymJ_{m,\ell}^{t}$ 
contains two biletters $\binom{i_1}{i_2}$ and $\binom{\ell}{\sigma_{\ell}}$ of forms $\binom{a_i}{a_j}$ and $\binom{\ell}{a_k}$, respectively. 
We define another permutation $\tau$ by exchanging $i_2$ and $\sigma_{\ell}$ in the bottom line. This procedure is reversible. 
By Lemma \ref{lemma:set}(i) it is easy to verify that $\tau\in \SymJ_{m,\ell'}^{t}$, where $\ell'=i_1\in  N|_{m} \cap A_i$. 
Thus, the transformation $\sigma\leftrightarrow \tau$ is an {\it involution} on $\sum_{\ell\in N|_{m}\cap A_{i}}\mathfrak{J}_{m,\ell}^{t}$. 
Hence,  \eqref{equ:eqsum0} holds.
\end{proof}
\medskip

Let $m=dn+h$ $(n\geq 2, \, h\in N|_d)$, $k\in N|_d$ and
$$\SymP_Y:=\SymJ_{m,m}^{t}, \quad
\SymP_Z:=\SymJ_{m,m-1}^{t}, \quad
\SymP_X:=\sum_{\ell\in N|_m\cap A_k}\SymJ_{m,\ell}^{t}.$$
The recurrence relations listed in Lemmas  
\ref{lemma:main3XYZ},
\ref{lemma:main3UVW},
\ref{lemma:main5},
\ref{lemma:main11XYZ},
\ref{lemma:main11UVW}
can be generated by Algorithm 1.
The procedure {\tt EvalAtoms(P,t,h,k)} appearing in Algorithm 1
evaluates
the cardinality of 
the set $\SymP:=\SymP_Y$, $\SymP_Z$ or $\SymP_X$ 
for each type $t$, and will be discussed in Section \ref{sec:Algorithm} (see Algorithm~2).

\begin{remark}
	By Step~2 the recurrence relations generated by Algorithms~1 and 2 are valid for $n\geq 2$. However, we can certify that they are also true for $n=1$ by using the method described in Sections \ref{sec:CountType} and \ref{sec:Algorithm}.

\end{remark}

\begin{program}[htp]\label{algo:rec}
\begin{verbatim}
for P in ['PX', 'PY', 'PZ']:
  for h in range(d):
    Val=0
    for k in range(d) if P=='PX' else range(1):
      for t in PossibleTypes(P,h,k):
        Val=Val+EvalAtoms(P,t,h,k)
    print P,h,k,Val
\end{verbatim}
\caption{Finding the recurrences}
\end{program}

\medskip

{\it Step 3. Counting permutations}. 
Throughout this step we fix $m=dn+h$ ($h\in N|_d$). 
Counting permutations from $\SymJ_{m,\ell}^{t}$ is lengthy;
it is made in several  substeps. 
We illustrate the entire calculations
by means of four well-selected examples,
using some {\it compressed and intuitive notations}. 
Then, we explain what those compressed notations mean in full detail.
The examples are given for $d=5$. We write $A,B,C,D,E$ instead of $A_0, A_1, A_2, A_3, A_4$ and $a,b,c,d,e$ instead of 
$a_0, a_1, a_2, a_3, a_4$, respectively.

\smallskip
\begin{example}\label{example1}
Consider $m=\ell=5n+1$ and the type $\String{adbca}$ which satisfies condition \eqref{cond:type}. We have
\begin{align*}
\mathfrak{J}_{5n+1,5n+1}^{adbca}
&\overset{w}{=}
  \left(
	\block{ \ch0\ch{\tilde5}\cch{1\spred0}\cch{1\spred5}}{\ch{e}\ch{a}\cch{e}\cch{e}}
	\bigg|
	\block{ \ch{1}\ch6\cch{1\spred1}}{\ch{d}\ch{d}\cch{d}}
	\bigg|
	\block{ \ch2\ch{\tilde7}\cch{1\spred2}}{\ch{c}\ch{b}\cch{c}}
	\bigg|
	\block{ \ch{\tilde3}\ch{8}\cch{1\spred3}}{\ch{c}\ch{b}\cch{b}}
	\bigg|
	\block{ \ch4\ch{9}\cch{1\spred4}}{\ch{a}\ch{a}\cch{a}}
 	\right)\\
&\overset{a}{=}
  \left(
	\block{ \ch0\ch{\tilde5}\cch{1\spred0}\cch{1\spred5}}{\ch{e}\ch{a}\cch{e}\cch{e}}
	\bigg|
	\block{ \ch{1}\ch6\cch{1\spred1}}{\ch{d}\ch{d}\cch{d}}
	\bigg|
	\block{ \ch2\ch{\tilde7}\cch{1\spred2}}{\ch{c}\ch{b}\cch{c}}
	\bigg|
	\block{ \ch{\tilde3}\ch{8}\cch{1\spred3}}{\ch{c}\ch{b}\cch{b}}
	\bigg|
	\block{ \ch4\ch{9}\cch{1\spred4}\cch{1\spred9}}{\ch{a}\ch{a}\cch{a}\cch{\underline{1\spred9}}}
 	\right)\\		
&\overset{e}{=}
\left(
	\block{ \ch0\cch{\tilde5}\cch{1\spred0}\cch{1\spred5}}{\ch{e}\cch{\underline{1\spred 9}}\cch{e}\cch{e}}
	\bigg|
	\block{ \ch{1}\ch6\cch{1\spred1}}{\ch{d}\ch{d}\cch{d}}
	\bigg|
	\block{ \ch2\ch{\tilde7}\cch{1\spred2}}{\ch{c}\ch{\underline{c}}\cch{c}}
	\bigg|
	\block{ \ch{\tilde3}\ch{8}\cch{1\spred3}}{\ch{\underline{b}}\ch{b}\cch{b}}
	\bigg|
	\block{ \ch4\ch{9}\cch{1\spred4}\cch{1\spred9}}{\ch{a}\ch{a}\cch{a}\cch{\underline{a}}}
 	\right)\\
&\overset{d}{=}
	\left(\block{ \ch0\cch{\tilde5}\cch{1\spred0}\cch{1\spred5}}{\ch{e}\cch{\underline{1\spred 9}}\cch{e}\cch{e}}\right)
\left(\block{ \ch{1}\ch6\cch{1\spred1}}{\ch{d}\ch{d}\cch{d}}\right)
\left(\block{ \ch2\ch{\tilde7}\cch{1\spred2}}{\ch{c}\ch{\underline{c}}\cch{c}}\right)
\left(\block{ \ch{\tilde3}\ch{8}\cch{1\spred3}}{\ch{\underline{b}}\ch{b}\cch{b}}\right)
	\left(\block{\ch4\ch{9}\cch{1\spred4}\cch{1\spred9}}{\ch{a}\ch{a}\cch{a}\cch{\underline{a}}}\right)\\
	&\overset{b}{=} Z_{n+1} \times Y_n\times X_n\times X_n\times Z_{n+1}.
\end{align*}
\end{example}

\begin{example}\label{example2}
Consider $m=5n+2$, $\ell=5n+1$ and the type $\String{dcbbaa}$ which satisfies condition \eqref{cond:type}. 
We have 
\begin{align*}
&\mathfrak{J}_{5n+2,5n+1}^{dcbbaa}\\
&\overset{w}{=}
\left(
	\block{ \ch0\ch{\tilde5}\cch{1\spred0}\cch{1\spred5}}{\ch{e}\ch{d}\cch{e}\cch{e}}
	\bigg|
	\block{ \ch{\tilde1}\ch6\cch{1\spred1}\cch{1\spred6}}{\ch{c}\ch{d}\cch{d}\cch{\underline{a}}}
	\bigg|
	\block{ \ch2\ch{\tilde7}\cch{1\spred2}}{\ch{c}\ch{b}\cch{c}}
	\bigg|
	\block{ \ch3\ch{8}\cch{1\spred3}}{\ch{b}\ch{b}\cch{b}}
	\bigg|
	\block{ \ch4\ch{9}\cch{1\spred4}}{\ch{a}\ch{a}\cch{a}}
 	\right)	\\
&\overset{a}{=}
\left(
	\block{ \ch0\ch{\tilde5}\cch{1\spred0}\cch{1\spred5}}{\ch{e}\ch{d}\cch{e}\cch{e}}
	\bigg|
	\block{ \ch{\tilde1}\ch6\cch{1\spred1}\cch{1\spred6}}{\ch{c}\ch{d}\cch{d}\cch{\underline{a}}}
	\bigg|
	\block{ \ch2\ch{\tilde7}\cch{1\spred2}}{\ch{c}\ch{b}\cch{c}}
	\bigg|
	\block{ \ch3\ch{8}\cch{1\spred3}\cch{1\spred8}}{\ch{b}\ch{b}\cch{b}\cch{\underline{1\spred8}}}
	\bigg|
	\block{ \ch4\ch{9}\cch{1\spred4}\cch{1\spred9}}{\ch{a}\ch{a}\cch{a}\cch{\underline{1\spred9}}}
 	\right)	\\
&\overset{e}{=}
   \left(
	\block{ \ch0\cch{\tilde5}\cch{1\spred0}\cch{1\spred5}}{\ch{e}\cch{\underline{1\spred9}}\cch{e}\cch{e}}
	\bigg|
	\block{ \ch{\tilde1}\ch6\cch{1\spred1}\cch{1\spred6}}{\ch{\underline{d}}\ch{d}\cch{d}\cch{\underline{1\spred8}}}
	\bigg|
	\block{ \ch2\ch{\tilde7}\cch{1\spred2}}{\ch{c}\ch{\underline{c}}\cch{c}}
	\bigg|
	\block{ \ch3\ch{8}\cch{1\spred3}\cch{1\spred8}}{\ch{b}\ch{b}\cch{b}\cch{\underline{b}}}
	\bigg|
	\block{ \ch4\ch{9}\cch{1\spred4}\cch{1\spred9}}{\ch{a}\ch{a}\cch{a}\cch{\underline{a}}}
 	\right)\\	
&\overset{d}{=}
\left(	\block{ \ch0\cch{\tilde5}\cch{1\spred0}\cch{1\spred5}}{\ch{e}\cch{\underline{1\spred9}}\cch{e}\cch{e}}\right)
\left(	\block{ \ch{\tilde1}\ch6\cch{1\spred1}\cch{1\spred6}}{\ch{\underline{d}}\ch{d}\cch{d}\cch{\underline{1\spred8}}}\right)
\left(	\block{ \ch2\ch{\tilde7}\cch{1\spred2}}{\ch{c}\ch{\underline{c}}\cch{c}}\right)
\left(	\block{ \ch3\ch{8}\cch{1\spred3}\cch{1\spred8}}{\ch{b}\ch{b}\cch{b}\cch{\underline{b}}}\right)
\left(	\block{ \ch4\ch{9}\cch{1\spred4}\cch{1\spred9}}{\ch{a}\ch{a}\cch{a}\cch{\underline{a}}}\right)\\
&\overset{b}{=}
Z_{n+1}\times X_n \times X_n \times Z_{n+1} \times Z_{n+1}.
\end{align*}
\end{example}
\begin{example}\label{example3}
Consider $m=5n+4$, $\ell\in C|_{n+1}$ and the type $\String{adcbac}$ which satisfies condition \eqref{cond:type}. 
We have 
\begin{align*}
&\sum_{\ell\in C|_{n+1}}\mathfrak{J}_{5n+4,\ell}^{adcbac}\\
&\overset{w}{=}
  \left(
	\block{ \ch0\ch{\tilde5}\cch{1\spred0}\cch{1\spred5}}{\ch{e}\ch{a}\cch{e}\cch{e}}
	\bigg|
	\block{ \ch{1}\ch6\cch{1\spred1}\cch{1\spred6}}{\ch{d}\ch{d}\cch{d}\cch{d}}
	\bigg|
	\block{ \ch2\ch{\tilde7}\cch{1\spred2}\cch{1\spred7}}{\ch{c}\ch{\underline{c}}\cch{c}\cch{c}}
	\bigg|
	\block{ \ch3\ch{8}\cch{1\spred3}\cch{1\spred8}}{\ch{b}\ch{b}\cch{b}\cch{b}}
	\bigg|
	\block{ \ch4\ch{9}\cch{1\spred4}}{\ch{a}\ch{a}\cch{a}}
 	\right)\\
&\overset{a}{=}
  \left(
	\block{ \ch0\ch{\tilde5}\cch{1\spred0}\cch{1\spred5}}{\ch{e}\ch{a}\cch{e}\cch{e}}
	\bigg|
	\block{ \ch{1}\ch6\cch{1\spred1}\cch{1\spred6}}{\ch{d}\ch{d}\cch{d}\cch{d}}
	\bigg|
	\block{ \ch2\ch{\tilde7}\cch{1\spred2}\cch{1\spred7}}{\ch{c}\ch{\underline{c}}\cch{c}\cch{c}}
	\bigg|
	\block{ \ch3\ch{8}\cch{1\spred3}\cch{1\spred8}}{\ch{b}\ch{b}\cch{b}\cch{b}}
	\bigg|
	\block{ \ch4\ch{9}\cch{1\spred4}\cch{1\spred9}}{\ch{a}\ch{a}\cch{a}\cch{\underline{1\spred9}}}
 	\right)\\
&\overset{e}{=}		
  \left(
	\block{ \ch0\cch{\tilde5}\cch{1\spred0}\cch{1\spred5}}{\ch{e}\ch{\underline{1\spred9}}\cch{e}\cch{e}}
	\bigg|
	\block{ \ch{1}\ch6\cch{1\spred1}\cch{1\spred6}}{\ch{d}\ch{d}\cch{d}\cch{d}}
	\bigg|
	\block{ \ch2\ch{\tilde7}\cch{1\spred2}\cch{1\spred7}}{\ch{c}\cch{\underline{c}}\cch{c}\cch{c}}
	\bigg|
	\block{ \ch3\ch{8}\cch{1\spred3}\cch{1\spred8}}{\ch{b}\ch{b}\cch{b}\cch{b}}
	\bigg|
	\block{ \ch4\ch{9}\cch{1\spred4}\cch{1\spred9}}{\ch{a}\ch{a}\cch{a}\cch{\underline{a}}}
 	\right)\\	
&\overset{d}{=}
  \left(\block{ \ch0\cch{\tilde5}\cch{1\spred0}\cch{1\spred5}}{\ch{e}\cch{\underline{1\spred9}}\cch{e}\cch{e}}\right)
\!\!
  \left(\block{ \ch{1}\ch6\cch{1\spred1}\cch{1\spred6}}{\ch{d}\ch{d}\cch{d}\cch{d}}\right)
\!\!
	\left(\block{ \ch2\ch{\tilde7}\cch{1\spred2}\cch{1\spred7}}{\ch{c}\ch{\underline{c}}\cch{c}\cch{c}}\right)
\!\!
	\left(\block{ \ch3\ch{8}\cch{1\spred3}\cch{1\spred8}}{\ch{b}\ch{b}\cch{b}\cch{b}}\right)
\!\!
	\left(\block{ \ch4\ch{9}\cch{1\spred4}\cch{1\spred9}}{\ch{a}\ch{a}\cch{a}\cch{\underline{a}}}\right)\\
&\overset{b}{=}Z_{n+1}\times Y_{n+1}\times X_{n+1}\times Y_{n+1}\times Z_{n+1}.
\end{align*}
\end{example}
\begin{example}\label{example4}
Consider $m=5n+1$, $\ell=5n$ and the type $\String{edcaab}$ which satisfies condition \eqref{cond:type}. 
We have 
\begin{align*}
&\mathfrak{J}_{5n+1,5n}^{edcaab}\\
&\overset{w}{=}
\left(
	\block{ \ch0\ch{5}\cch{1\spred0}\cch{1\spred5}}{\ch{e}\ch{e}\cch{e}\cch{\underline{b}}}
	\bigg|
	\block{ \ch{1}\ch6\cch{1\spred1}}{\ch{d}\ch{d}\cch{d}}
	\bigg|
	\block{ \ch2\ch{7}\cch{1\spred2}}{\ch{c}\ch{c}\cch{c}}
	\bigg|
	\block{ \ch{\tilde3}\ch{8}\cch{1\spred3}}{\ch{a}\ch{b}\cch{b}}
	\bigg|
	\block{ \ch4\ch{9}\cch{1\spred4}}{\ch{a}\ch{a}\cch{a}}
 	\right)	\\
&\overset{a}{=}
\left(
	\block{ \ch0\ch{5}\cch{1\spred0}\cch{1\spred5}}{\ch{e}\ch{e}\cch{e}\cch{\underline{b}}}
	\bigg|
	\block{ \ch{1}\ch6\cch{1\spred1}}{\ch{d}\ch{d}\cch{d}}
	\bigg|
	\block{ \ch2\ch{7}\cch{1\spred2}}{\ch{c}\ch{c}\cch{c}}
	\bigg|
	\block{ \ch{\tilde3}\ch{8}\cch{1\spred3}}{\ch{a}\ch{b}\cch{b}}
	\bigg|
	\block{ \ch4\ch{9}\cch{1\spred4}\cch{1\spred9}}{\ch{a}\ch{a}\cch{a}\cch{\underline{1\spred9}}}
 	\right)	\\
&\overset{e}{=}
\left(
	\block{ \ch0\ch{5}\cch{1\spred0}\cch{1\spred5}}{\ch{e}\ch{e}\cch{e}\cch{\underline{1\spred9}}}
	\bigg|
	\block{ \ch{1}\ch6\cch{1\spred1}}{\ch{d}\ch{d}\cch{d}}
	\bigg|
	\block{ \ch2\ch{7}\cch{1\spred2}}{\ch{c}\ch{c}\cch{c}}
	\bigg|
	\block{ \ch{\tilde3}\ch{8}\cch{1\spred3}}{\ch{\underline{b}}\ch{b}\cch{b}}
	\bigg|
	\block{ \ch4\ch{9}\cch{1\spred4}\cch{1\spred9}}{\ch{a}\ch{a}\cch{a}\cch{\underline{a}}}
 	\right)	\\
&\overset{d}{=}
\left(\block{ \ch0\ch{5}\cch{1\spred0}\cch{1\spred5}}{\ch{e}\ch{e}\cch{e}\cch{\underline{1\spred9}}}\right)
\left(	\block{ \ch{1}\ch6\cch{1\spred1}}{\ch{d}\ch{d}\cch{d}}\right)
\left(	\block{ \ch2\ch{7}\cch{1\spred2}}{\ch{c}\ch{c}\cch{c}}\right)
\left(	\block{ \ch{\tilde3}\ch{8}\cch{1\spred3}}{\ch{\underline{b}}\ch{b}\cch{b}}\right)
\left(	\block{ \ch4\ch{9}\cch{1\spred4}\cch{1\spred9}}{\ch{a}\ch{a}\cch{a}\cch{\underline{a}}}\right)\\
&\overset{b}{=}
Y_{n}\times Y_n \times Y_n \times X_{n} \times Z_{n+1}.
\end{align*}
\end{example}

{\it Notation 1.} In the above compressed writing, the letter $w,a,e,d,b$ over the symbol $``="$ means that the equality is obtained by substep $3(w)$, $3(a)$, $3(e)$, $3(d)$, $3(b)$ respectively. 

\smallskip

{\it Notation 2.} In the compressed writing the integer $n$ is represented by
the explicit value $3$. Hence, the second block in the first equality in Example \ref{example1} has the following meaning:
$$
\bigg|\block{ \ch1\ch{6}\cch{1\spred1}}{\ch{d}\ch{d}\cch{d}}\bigg|
:=
\bigg|\block{ \ch1\ch{6}\cch{1\spred1}  \cch{1\spred6} \cdots 5n-4}{\ch{d}\ch{d}\cch{d} \cch{d} \cdots \quad \cch{d}\quad }\bigg|.
$$
Also, the added biletter $\binom {19}{\underline{19}}$  (see Substep 3(a))
in the second equality in Example \ref{example1} 
means $\binom {5n+4}{\underline{5n+4}}$.

\smallskip

{\it Substep 3(w). Rewrite the set.}
For each permutation $\sigma$ from $\SymJ_{m,\ell}^t$,
we reorder the biletters of $\sigma$ such that
$\binom {i}{\sigma_i}$ is on the left of $\binom j{\sigma_j}$ if
$i\mod d < j\mod d$, or if $i\equiv j\mod d$ and $i<j$.
Then, we replace each letter $y\in a_{k}$ in the bottom line by $a_{k}$.
To facilitate readability, vertical bars are inserted 
between the biletters $\binom {i}{\sigma_i}$ and
$\binom {j}{\sigma_j}$ such that $i\not\equiv j \pmod d$. We get a biword $w$, denoted by $\rho(\sigma)=w$, called {\it shape} of $\sigma$.

Applying this operation on the following  permutation $\sigma\in 
\mathfrak{J}_{5n+1,5n+1}^{adbca}$ considered in Example \ref{example1} 
\begin{equation}
\sigma=
\left(
\block{\ch{0}\ch{1}\ch{2}\ch{3}\cch{4}\cch{5}\cch{6}\ch{7}\cch{8}\cch{9}\cch{1\spred0}\cch{1\spred1}\cch{1\spred2}\cch{1\spred3}\cch{1\spred4}\cch{1\spred5}}
{\ch{4}\ch{3}\ch{2}\ch{7}\cch{1\spred5}\cch{0}\cch{1\spred3}\ch{1}\cch{1\spred1}\cch{1\spred0}\cch{1\spred4}\cch{8}\cch{1\spred2}\cch{6}\cch{5}\cch{9}}
\right), 
\end{equation}
we get the shape $\rho(\sigma)=w$, where
\begin{equation}  
	w=
\left(
	\block{ \ch0\ch{5}\cch{1\spred0}\cch{1\spred5}}{\ch{e}\ch{a}\cch{e}\cch{e}}
	\bigg|
	\block{ \ch{1}\ch6\cch{1\spred1}}{\ch{d}\ch{d}\cch{d}}
	\bigg|
	\block{ \ch2\ch{7}\cch{1\spred2}}{\ch{c}\ch{b}\cch{c}}
	\bigg|
	\block{ \ch{3}\ch{8}\cch{1\spred3}}{\ch{c}\ch{b}\cch{b}}
	\bigg|
	\block{ \ch4\ch{9}\cch{1\spred4}}{\ch{a}\ch{a}\cch{a}}
 	\right).
\end{equation}

{\it Notation 3.} In the compressed writing, the above shape $w$ represents also the set $\rho^{-1}(w)$  of all the permutations $\sigma$ such that $\rho(\sigma)=w$.

Each permutation $\sigma\in \SymJ_{5n+1,5n+1}^{adbca}$ contains exactly 
three unsociable biletters of form $\binom aa,\binom cb, \binom dc$, denoted by $\binom{i_0}{j_0},\binom{i_1}{j_1},\binom{i_2}{j_2},$ respectively. 
So that, for example, in the block
$$
\bigg|\block{ \ch2\ch{7}\cch{1\spred2}}{\ch{c}\ch{b}\cch{c}}\bigg|
$$
there is exactly one letter $\String{b}$ in the bottom line. All other letters are~$\String{c}$. However, the position of the letter $\String{b}$ is not fixed.
The shape of another permutation may contain the block
$$
\bigg|\block{ \ch2\ch{7}\cch{1\spred2}}{\ch{b}\ch{c}\cch{c}}\bigg|
\text{\qquad or\qquad }
\bigg|\block{ \ch2\ch{7}\cch{1\spred2}}{\ch{c}\ch{c}\cch{b}}\bigg|.
$$

\smallskip
{\it Notation 4.} The underlined bileters $\binom {i}{\underline {a_j}}$ in the shape of a permutation~$\sigma$ means that there is no constraint $i+\sigma_i\in J$ for the corresponding biletters $\binom {i}{\sigma_i}$ of~$\sigma$. 
All other biletters of~$\sigma$ must satisfy 
the latter constraint.

In the first equality of each calculation, there is no underlined biletter 
if $m=\ell$ (Example \ref{example1}) or exactly one underlined biletter if $0\leq \ell\leq  m-1$ (Examples \ref{example2} and \ref{example3}). In the latter case, the underscore sign 
indicates
the position of $\ell$.

\smallskip
{\it Notation 5.}
The shape $w$, with a tilde sign $\tilde{}$ over a biletter $\binom {\tilde{\imath}}{a_j}$,
represents the sum of all shapes $w'$ which are obtained from $w$
by moving the letter $a_j$, including the underscore sign if it is underlined, to other non-underlined position in the block. For examples, we write (see Example~\ref{example1})
$$
\bigg|\block{ \ch2\ch{\tilde7}\cch{1\spred2}}{\ch{c}\ch{b}\cch{c}}\bigg|
:=
\bigg|\block{ \ch2\ch{7}\cch{1\spred2}}{\ch{b}\ch{c}\cch{c}}\bigg|
+
\bigg|\block{ \ch2\ch{7}\cch{1\spred2}}{\ch{c}\ch{b}\cch{c}}\bigg|
+
\bigg|\block{ \ch2\ch{7}\cch{1\spred2}}{\ch{c}\ch{c}\cch{b}}\bigg|,
$$
and (see Example \ref{example2})
$$
\bigg|
	\block{ \ch{\tilde1}\ch6\cch{1\spred1}\cch{1\spred6}}{\ch{c}\ch{d}\cch{d}\cch{\underline{a}}} \bigg|
:=
\bigg|
	\block{ \ch{1}\ch6\cch{1\spred1}\cch{1\spred6}}{\ch{c}\ch{d}\cch{d}\cch{\underline{a}}} \bigg|
	+
\bigg|
	\block{ \ch{1}\ch6\cch{1\spred1}\cch{1\spred6}}{\ch{d}\ch{c}\cch{d}\cch{\underline{a}}} \bigg|
	+
\bigg|
	\block{ \ch{1}\ch6\cch{1\spred1}\cch{1\spred6}}{\ch{d}\ch{d}\cch{c}\cch{\underline{a}}} \bigg|,
$$
$$
	\bigg|
	\block{ \ch{\tilde1}\ch6\cch{1\spred1}\cch{1\spred6}}{\ch{\underline{d}}\ch{d}\cch{d}\cch{\underline{1\spred8}}}
	\bigg|
	:=
	\bigg|
	\block{ \ch{1}\ch6\cch{1\spred1}\cch{1\spred6}}{\ch{\underline{d}}\ch{d}\cch{d}\cch{\underline{1\spred8}}}
	\bigg|
	+
	\bigg|
	\block{ \ch{1}\ch6\cch{1\spred1}\cch{1\spred6}}{\ch{d}\ch{\underline{d}}\cch{d}\cch{\underline{1\spred8}}}
	\bigg|
	+
	\bigg|
	\block{ \ch{1}\ch6\cch{1\spred1}\cch{1\spred6}}{\ch{d}\ch{d}\cch{\underline{d}}\cch{\underline{1\spred8}}}
	\bigg|.
$$
Notice that there is at most one tilde in each block by Lemma \ref{lemma:eqsum0}

\smallskip

{\it Substep 3(a). Add biletters.}
For each $\sigma\in \SymJ_{dn+h,\ell}^{t}$, we add all biletters 
$\binom ii$ such that $\max\{dn+h,dn+d-h\}\le i \le dn+d-1$. 
Thus,
the number of occurrences of $a_{j}$ in the bottom row becomes the same as 
the number of occurrences of $a_{d-j-1}$ for any $j\in N|_{d}$.

For instance, the bottom row of the right-hand side of $\overset {w}{=}$ in Example~\ref{example1} contains $4\times a, 3\times b, 3\times c, 3\times d, 3\times e$.
By adding the biletter $\binom {19}{19}$ to the shape 
the number of occurrences of $a$ in the bottom row becomes the same as 
the number of occurrences of~$e$ (since $19$ is also an $\String{e}$). 
The added biletter in the shape
is still represented by $\binom {19}{19}$, instead of 
$\binom {19}{e}$. Notice that it is underlined (see Notation 4).
\smallskip

{\it Substep 3(e). Exchange.} 
Consider all the biletters of the permutation~$\sigma$,
which
are unsocial, or which were added in Substep 3(a),
or still which have the specific form $\binom {\ell}{\underline{a_k}}$ with $0\leq \ell\leq m-1$. 
Exchange the bottom letters of those biletters
in such a way that
all the biletters
will become friendly. 
In most of the cases, each block contains zero or one bad biletter. The only exception is 
the block containing the specific form $\binom {\ell}{\underline{a_k}}$ 
with $\ell=m-1$,
and another unsocial biletter $\binom {i}{a_j}$.
In such a case we put the appropriate {\it explicit letter}, which was added 
in Substep 3(a), under the letter~$\ell$ when the exchange was made.
The whole procedure is reversible.

In Examples \ref{example1} and \ref{example2}, the exchanges of the bad biletters are realized respectively as follows:
$$
\left(\cdots
\block{ \ch{\tilde5}  \ch{\tilde7} \ch{\tilde3}  \cch{1\spred9}}{\ch{a} \ch{b} \ch{c}  \cch{\underline{1\spred9}} }
 \cdots \right)
\qquad
\mapsto
\qquad
\left(\cdots\block{ \cch{\tilde5}  \ch{\tilde7} \ch{\tilde3}  \cch{1\spred9}}{\cch{\underline{1\spred9}} \ch{\underline{c}} \ch{\underline{b}} \cch{\underline{a}}  } \cdots \right)
$$
$$
\left(\cdots\block{ \ch{\tilde5}  \ch{\tilde1} \cch{1\spred6} \ch{\tilde7} \cch{1\spred8} \cch{1\spred9}}{\ch{d} \ch{c} \cch{\underline a} \ch{b}  \cch{\underline{1\spred8}}\cch{\underline{1\spred9}} } \cdots \right)
\qquad
\mapsto
\qquad
\left(\cdots\block{ \cch{\tilde5}  \ch{\tilde1} \cch{1\spred6} \ch{\tilde7} \cch{1\spred8} \cch{1\spred9}}{
\cch{\underline{1\spred9}}
\ch{\underline{d}}    \cch{\underline{1\spred8}} \ch{\underline{c}} \cch{\underline{b}} \cch{\underline{a}}  } \cdots \right)
$$
In the second example, the block 
{\small $\bigl|{\ch{\tilde1} \cch{1\spred6}\atop 
\ch{c} \cch{\underline a}}\bigr|$}
contains two bad biletters.
We put the {\it explicit letter} $18$ instead of the symbol `$d$' under the letter $\ell=16$.

\smallskip

{\it Substep 3(d). Decomposition.}
After Substep 3(e)  Exchange, 
the set $\SymJ_{dn+h,\ell}^{t}$ is decomposed, in a natural way,
into the 
Cartesian product of $d$ sets 
of biwords, which are called {\it atoms} in the sequel.
According to the situation of the tilde and underscore signs, the atoms
are classified into six families:
\begin{align*}
	(i):\quad &\left(\block{ \ch{1}\ch6\cch{1\spred1}}{\ch{d}\ch{d}\cch{d}}\right),
  \left(\block{ \ch{1}\ch6\cch{1\spred1}\cch{1\spred6}}{\ch{d}\ch{d}\cch{d}\cch{d}}\right),
	\left(\block{ \ch3\ch{8}\cch{1\spred3}\cch{1\spred8}}{\ch{b}\ch{b}\cch{b}\cch{b}}\right),
\left(	\block{ \ch2\ch{7}\cch{1\spred2}}{\ch{c}\ch{c}\cch{c}}\right);\\
	(i'): \quad &\left(\block{ \ch0\ch{5}\cch{1\spred0}\cch{1\spred5}}{\ch{e}\ch{e}\cch{e}\cch{\underline{1\spred9}}}\right); \\
(ii):\quad &\left(	\block{ \ch4\ch{9}\cch{1\spred4}\cch{1\spred9}}{\ch{a}\ch{a}\cch{a}\cch{\underline{a}}}\right),
\left(	\block{ \ch3\ch{8}\cch{1\spred3}\cch{1\spred8}}{\ch{b}\ch{b}\cch{b}\cch{\underline{b}}}\right);\\
(ii'):\quad &\left(\block{ \ch0\cch{\tilde5}\cch{1\spred0}\cch{1\spred5}}{\ch{e}\cch{\underline{1\spred 9}}\cch{e}\cch{e}}\right);\\
(iii):\quad &\left(	\block{ \ch2\ch{\tilde7}\cch{1\spred2}}{\ch{c}\ch{\underline{c}}\cch{c}}\right),
\left(\block{ \ch{\tilde3}\ch{8}\cch{1\spred3}}{\ch{\underline{b}}\ch{b}\cch{b}}\right),
	\left(\block{ \ch2\ch{\tilde7}\cch{1\spred2}\cch{1\spred7}}{\ch{c}\ch{\underline{c}}\cch{c}\cch{c}}\right);\\
	(iii'):\quad &\left(	\block{ \ch{\tilde1}\ch6\cch{1\spred1}\cch{1\spred6}}{\ch{\underline{d}}\ch{d}\cch{d}\cch{\underline{1\spred8}}}\right).
\end{align*}
It suffices to count the permutations in each atom.
\smallskip

{\it Substep 3(b). Beta transformation.}
The cardinalities of the atoms can be derived by means of the transformation $\beta$ defined in \eqref{equ:beta}.
We discuss the method according to the classification given in Substep 3(d). 

$(i)$ The atom
\begin{equation}\label{atom:i}
\SymA_{0}=
	\begin{pmatrix}
	1 & 6 & 11				\\
	d & d & d	
	\end{pmatrix}
\end{equation}
represents the set 
\begin{equation*}
	 \left\{	\left. \begin{pmatrix}
			 1 & 6 & 11 	\\
			 \tau_{1} & \tau_{6} & \tau_{11}
	\end{pmatrix}\ \right|\  
	\begin{matrix}
		\{\tau_{1}, \tau_{6}, \tau_{11}\} = D|_{n}  \\
    i+\tau_{i}\in J \text{ for }i \in B|_{n} 
\end{matrix}
\right\}.
\end{equation*}

If $i+\sigma_i\in A_{d-1}$, then
$$
i+\sigma_i \in J 
\Longleftrightarrow  
\beta(i)+\beta(\sigma_i)=\beta(i+\sigma_i)\in \bar{J}
$$
by Lemma \ref{lemma:set} (iii). Applying the transformation $\beta$ to each letter in the top and bottom rows of each element $\tau$ of $\SymA_{0}$, we get a permutation $\lambda$ from $\SymJ_{n,n}$:
\begin{equation*}
	 \left\{	\left.
 \beta\begin{pmatrix}
			 1 & 6 & 11 	\\
			 \tau_{1} & \tau_{6} & \tau_{11}
	\end{pmatrix}\ 
 =
 \begin{pmatrix}
			 0 & 1 & 2 	\\
			 \lambda_{1} & \lambda_{2} & \lambda_{3}
	\end{pmatrix}\ \right|\  
	\begin{matrix}
		\{\lambda_{0}, \lambda_{1}, \lambda_{2}\} = N|_{n}  \\
    i+\lambda_{i}\in \bar{J} \text{ for }i \in N|_{n} 
\end{matrix}
\right\}.
\end{equation*}
The above transformation is reversible and the atom 
$\SymA_{0}$ is in bijection with $\bar{\SymJ}_{n,n}$. 
Thus $\#\SymA_{0}=\#\bar{\SymJ}_{n,n}=\bar{Y}_{n}$.
For example, the second factor appearing in the right-hand side of the equality $\overset{b}{=}$ in Example~\ref{example1}, is equal to $\bar Y_n=Y_n$.

{\it Notation 6.} 
In the compressed writing, a set symbol may designate also the cardinality of the set, if necessary. For example, we may write $\Sym_4=24$.

\smallskip
$(i')$ The atom
\begin{equation}\label{atom:i'}
\SymA_{1}=
	\begin{pmatrix}
		0 & 5 & 10	& 15			\\
		e & e & e	  &\underline{19}
	\end{pmatrix}
\end{equation}
represents the set 
\begin{equation*}
	 \left\{	\left. \begin{pmatrix}
		0 & 5 & 10	& 15			\\
			 \tau_{0} & \tau_{5} & \tau_{10} & 19
	\end{pmatrix}\ \right|\  
	\begin{matrix}
		\{\tau_{0}, \tau_{5}, \tau_{10}\} = E|_{n}  \\
    i+\tau_{i}\in J \text{ for }i \in A|_{n} 
\end{matrix}
\right\}.
\end{equation*}
Thus, it has the same cardinality of the atom $\SymA_0$ defined in \eqref{atom:i}.

\smallskip

$(ii)$ The atom
\begin{equation}\label{atom:ii}
\SymA_{2}=
	\begin{pmatrix}
	3 & 8 & 13 & 18				\\
	b & b & b  &	\underline{b}
	\end{pmatrix}
\end{equation}
is meant to be the set 
\begin{equation*}
	 \left\{	\left. \begin{pmatrix}
			 3 & 8 & 13 & 18 	\\
			 \tau_{3} & \tau_{8} & \tau_{16} & \underline{\tau_{18}}
	\end{pmatrix}\ \right|\  
	\begin{matrix}
		\{\tau_{3}, \tau_{8}, \tau_{13},\tau_{18}\} = B|_{n+1}  \\
    i+\tau_{i}\in J \text{ for }i \in D|_{n} 
\end{matrix}
\right\}.
\end{equation*}
Applying the transformation $\beta$ to each letter in each element $\sigma$ in the atom $\SymA_{2}$, we get a permutation $\lambda$ from $\bar{\SymJ}_{n+1,n}$: 
\begin{equation*}
	 \left\{	\left.
  \begin{pmatrix}
			 0 & 1 & 2 	& 3\\
			 \lambda_{0} & \lambda_{1} & \lambda_{2} & \lambda_{3}
	\end{pmatrix}\ \right|\  
	\begin{matrix}
		\{\lambda_{0}, \lambda_{1}, \lambda_{2}, \lambda_{3}\} = N|_{n+1}  \\
		i+\lambda_{i}\in \bar{J} \ \text{ for }i \in N|_{n} 
\end{matrix}
\right\}.
\end{equation*}
The transformation is reversible, so that $\SymA_{2}$ is in bijection 
with $\bar{\SymJ}_{n+1,n}$. 
Hence, $\#\SymA_{2}=\#\bar{\SymJ}_{n+1,n}=\bar{Z}_{n+1}$.

\smallskip
$(ii')$ The atom
\begin{equation}\label{atom:ii'}
\SymA_{3}=
	\begin{pmatrix}
	0 & \tilde{5} & 10			& 15	\\
	e & \underline{19}& e	& e	\\
	\end{pmatrix}
\end{equation}
represents the set 
\begin{equation*}
	 \left\{\!	\left. \begin{pmatrix}
			 0 & 5 & 10 &15	\\
			 \tau_{0} & \tau_{5} & \tau_{10}	 & \tau_{15}
	\end{pmatrix}\ \! \right|
	\begin{matrix}
		\{\tau_{0}, \tau_{5}, \tau_{10}, \tau_{15}\} = E|_{n+1}  \\ 
		i+\tau_i \in J \text{ for $i\in A|_{n+1}$ such that $\tau_i\not=5n\!+\!4$}
\end{matrix}
\right\}.
\end{equation*}
By inverting the top and bottom rows of each biword, the above set becomes
\begin{equation*}
	 \left\{	\left. \begin{pmatrix}
			 4 & 9 & 14 &19	\\
			 \rho_{4} & \rho_{9} & \rho_{14}	 & \underline{\rho_{19}}
	\end{pmatrix}\ \right|\  
	\begin{matrix}
		\{\rho_{4}, \rho_{9}, \rho_{14}, \rho_{19}\} = A|_{n+1}  \\ 
		i+\rho_i \in J \text{ for  $i\in E|_{n}$}
\end{matrix}
\right\},
\end{equation*}
which is equal to the atom 
\begin{equation}
	\begin{pmatrix}
	4 & 9 & 14 & 19	\\
	a & a & a & \underline a							\\
	\end{pmatrix}
\end{equation}
already studied in $(ii)$.

\smallskip

$(iii)$
The atom
\begin{equation}\label{atom:iii}
	\SymA_4=\begin{pmatrix}
	\tilde3 & 8 & 13				\\
	\underline{b} & b  &b	
	\end{pmatrix}
\end{equation}
represents the set
\begin{equation*}
	\sum_{r\in D|_{n}}	 \left\{	\left. \begin{pmatrix}
			  3 & 8 & 13 	\\
			  \tau_{3} & \tau_{8} & \tau_{13}
	\end{pmatrix}\ \right|\  
	\begin{matrix}
		\{\tau_{3}, \tau_{8}, \tau_{13}\} = B|_{n}  \\
    i+\tau_{i}\in J \text{ for $i \in D|_{n}$ such that $i\neq r$}
\end{matrix}
\right\}.
\end{equation*}
Applying the transformation $\beta$,  the latter set becomes
\begin{equation*}
	\sum_{r=0}^{n-1}	 \left\{	\left. \begin{pmatrix}
			  0 & 1 & 2 	\\
			  \lambda_{0} & \lambda_{1} & \lambda_{2}
	\end{pmatrix}\ \right|\  
	\begin{matrix}
		\{\lambda_{0}, \lambda_{1}, \lambda_{2}\} = N|_{n}  \\
    i+\tau_{i}\in \bar{J} \text{ for $i \in N|_{n}$ such that $i\neq r$}
\end{matrix}
\right\}.
\end{equation*}
Hence, 
\begin{equation}\label{equ:iii}
\SymA_{4}=
	\sum\limits_{r=0}^{n-1}\#\bar{\SymJ}_{n,i}=\bar{X}_{n}.
\end{equation}
\smallskip

$(iii')$
Similar to $(i')$, we have
$$
\left(	\block{ \ch{\tilde1}\ch6\cch{1\spred1}\cch{1\spred6}}{\ch{\underline{d}}\ch{d}\cch{d}\cch{\underline{1\spred8}}}\right)
=
\left(	\block{ \ch{\tilde1}\ch6\cch{1\spred1}}{\ch{\underline{d}}\ch{d}\cch{d}
}\right),
$$
which was already studied in $(iii)$.


\section{Algorithm for evaluating the atoms}\label{sec:Algorithm}

Keep the same notations as in Section \ref{sec:CountType},
in particular, $m=dn+h$ $(h\in N|_d)$.
Let $k\in N|_d$. For simplicity, 
we write 
$$\SymP_Y:=\SymJ_{m,m}^{t}, \quad
\SymP_Z:=\SymJ_{m,m-1}^{t}, \quad
\SymP_X:=\sum_{\ell\in N|_m\cap A_k}\SymJ_{m,\ell}^{t}.$$
For each type $t$ the cardinality of 
the set $\SymP:=\SymP_Y$, $\SymP_Z$ or $\SymP_X$, is evaluated 
by the substeps $3(w)$, $3(a)$, $3(e)$, $3(d)$, $3(b)$, which are fully described in Section \ref{sec:CountType}.  As a consequence, the latter cardinality is equal to the product of $d$ factors (see Examples \ref{example1}--\ref{example4})
corresponding to the $d$ atoms respectively.
In this section, we show that the substeps in Step~3 can be combined onto one super-step. In fact, each factor can be 
evaluated directly by using a prefabricated dictionary.


\begin{defi}\label{def:mu}
Let $i\in N|_d$ be a fixed integer. We define several parameters depending on
$i, k, d, m, t$, where $t=s_0s_1\ldots s_{d-1}$ (if $\SymP=\SymP_Y$) or 
$s_0s_1\ldots s_{d-1}s_d$ (if $\SymP=\SymP_X$ or $\SymP_Z$):
\begin{align*}
\eta_{0}&= 
\begin{cases} 
1,  &\text{if\ } i+1\leq h, \\ 
0, &\text{otherwise}\,; 
\end{cases}\\
\eta_{1}&= 
\begin{cases} 
1,  &\text{if\ } d-i\leq h, \\ 
0, &\text{otherwise}\,; 
\end{cases}\\
\eta_{2}&= 
\begin{cases} 
1,  &\text{if\ } s_i=a_{d-i-1}, \\ 
0, &\text{otherwise}\,;
\end{cases}\\
\eta_{3}&= 
\begin{cases} 
	1,  &\text{if\ } \SymP\not=\SymP_Y \text{\ and \ } s_d=a_{d-i-1}, \\ 
0, &\text{otherwise}\,. 
\end{cases}\\
\nu &= 
\begin{cases} 
	\String{Z}, &\text{if\ } \SymP=\SymP_Z  \text{\ and\ }  m-1\in A_i ,\\
	\String{X}, &\text{if\ } \SymP=\SymP_X  \text{\ and\ }  k=i ,\\
	\String{G}, &\text{otherwise\ }; 
\end{cases}\\
\mu_i &= 
\begin{cases} 
	\Psi_Z(\eta_0, \eta_1, \eta_2, \eta_3), &\text{if\ } \nu=\String{Z},\\
	\Psi_X(\eta_0, \eta_1, \eta_2, \eta_3), &\text{if\ } \nu=\String{X},\\
	\Psi_G(\eta_0, \eta_1, \eta_2), &\text{if\ } \nu=\String{G}, 
\end{cases}
\end{align*}
where the explicit values of the functions 
$\Psi_Z, \Psi_X, \Psi_G$ are given in Table~\ref{tab:Psi}.
\begin{table}[!ht]\label{tab:Psi}
$$
\begin{array}{|c|cccccccc|}
\hline
\eta & 000 & 001 & 010 & 011 & 100 & 101 & 110 & 111\\
\Psi_G(\eta)&\bar{X}_{n} & \bar{Y}_{n} & 0 & \bar{Z}_{n+1} & \bar{Z}_{n+1} & 0 & \bar{X}_{n+1} & \bar{Y}_{n+1}\\
\hline
\eta & 0000 & 0010 & 0100 & 0110 & 1000 & 1010 & 1100 & 1110\\
\Psi_Z(\eta)& 0 & \bar{Z}_{n} & 0 & 0 & \bar{X}_{n} & \bar{Y}_{n} & 0 & \bar{Z}_{n+1}\\
\hline
\eta & 0001 & 0011 & 0101 & 0111 & 1001 & 1011 & 1101 & 1111\\
\Psi_Z(\eta)  & 0 & \bar{Z}_{n} & 0 & 0 & \bar{X}_{n} & 0 & 0 & \bar{Z}_{n+1}\\
\hline
\eta & 0000 & 0010 & 0100 & 0110 & 1000 & 1010 & 1100 & 1110\\
\Psi_X(\eta)& 0 & \bar{X}_{n} & 0 & 0 & 0 & \bar{Z}_{n+1} & 0 & \bar{X}_{n+1}\\
\hline
\eta & 0001 & 0011 & 0101 & 0111 & 1001 & 1011 & 1101 & 1111\\
\Psi_X(\eta)& 0 & \bar{X}_{n} & 0 & 0 & 0 & 0 & 0 & \bar{X}_{n+1}\\
\hline
\end{array}
$$
\caption{Explicit values of the functions $\Psi_Z, \Psi_X, \Psi_G$}
\end{table}
\end{defi}
\vskip -1em

Notice that each permutation contains 
$n+\eta_0$ (resp. $n+\eta_1$) letters in $A_i$ (resp. in $A_{d-i-1}$).

\begin{example}
Consider $\SymP=\SymJ_{5n+1, 5n+1}^{adbca}$, studied in Example \ref{example1}.
In this case, $d=5, m=5n+1, h=1, \ell=m=5n+1, t=s_0s_1s_2s_3s_4=\String{adbca}$.
For $i=1$ we have $\eta_0=0, \eta_1=0, \eta_2=1$. Hence, 
$\mu_1= \Psi_G(0,0,1) = \bar Y_n$.
\end{example}

\begin{example}\label{examplePsi2}
Consider $\SymP=\SymJ_{5n+2, 5n+1}^{dcbbaa}$, studied in Example \ref{example2}.
In this case, $d=5, m=5n+2, h=2, \ell=m-1=5n+1\in A_{1}, t=s_0s_1s_2s_3s_4=\String{dcbbaa}$.
For $i=1$ we have $\eta_0=1, \eta_1=0, \eta_2=0, \eta_3=0$. So that  $\mu_1= \Psi_Z(1,0,0,0) = \bar X_n$.
\end{example}
\begin{example}
Consider $\SymP=\sum_{\ell\in C|_{n+1}}\SymJ_{5n+4, \ell}^{adcbac}$, studied in Example \ref{example3}.
In this case, $d=5, m=5n+4, h=4, \ell\in A_{2}, t=s_0s_1s_2s_3s_4=\String{adcbac}$.
For $i=2$ we have $\eta_0=1, \eta_1=1, \eta_2=1, \eta_3=1$ and  $\mu_2= \Psi_X(1,1,1,1) = \bar X_{n+1}$.
\end{example}

\begin{thm}\label{thm:atom:Psi}
With the above notations, the cardinality of 
the set $\SymP:=\SymP_Y, \SymP_Z, \SymP_X$ is equal to
\begin{equation}\label{equ:mu}
\#\SymP=\mu_0 \times \mu_1 \times \mu_2 \times \cdots \times \mu_{d-1}.
\end{equation}
\end{thm}
\smallskip

For example, the 
set $\SymP=\sum_{\ell\in C|_{n+1}}\SymJ_{5n+4, \ell}^{adcbac}$, 
studied in Example \ref{example3}, is evaluated by means of Theorem \ref{thm:atom:Psi} as follows:
\begin{align*}
&\quad \sum_{\ell\in C|_{n+1}}\mathfrak{J}_{5n+4,\ell}^{adcbac}\\
 &=\mu_0\, \mu_1\, \mu_2\, \mu_3\, \mu_4\\
&=
\Phi_G(1,0,0)\, \Phi_G(1,1,1)\, \Psi_X(1,1,1,1)\, \Psi_G(1,1,1)\, \Phi_G(0,1,1)\\
&=Z_{n+1}\,Y_{n+1}\, X_{n+1}\, Y_{n+1}\, Z_{n+1}.
\end{align*}

\goodbreak

By Theorem \ref{thm:atom:Psi}, the 
procedure {\tt EvalAtoms(P,t,h,k)} figured in Algorithm 1,
which
evaluates
the cardinality of 
the set $\SymP:=\SymP_Y$, $\SymP_Z$ or $\SymP_X$ 
for each type $t$, is described in Algorithm 2.

\begin{program}[htp]\label{algo:atom}
\begin{verbatim}
def EvalAtoms(P,t,h,k):
  Prod=1
  for i in Ch:
    nu='G'
    if P=='PZ' and i==(h+d-1)%d: nu='Z'
    if P=='PX' and i==k:  nu='X'
    eta=(i+1<=h, d-i<=h, t[i]==d-i-1)
    if nu=='X' or nu=='Z': eta=eta+(t[d]==d-i-1,)
    Prod=Prod*Psi(nu, eta)
  return Prod
\end{verbatim}
\caption{Evaluating the atoms}
\end{program}

\begin{proof}[Proof of Theorem \ref{thm:atom:Psi}]
When we speak of {\it case}, we refer to a tuple  $(\nu=\String{G}, \eta_0, \eta_1, \eta_2)$,
$(\nu=\String{Z}, \eta_0, \eta_1, \eta_2, \eta_3)$ 
or $(\nu=\String{Z}, \eta_0, \eta_1, \eta_2, \eta_3)$,
which depends on $i,k,d,m,t,\SymP$ by Definition \ref{def:mu}.
The case is reproduced without non-significant symbols. For example,
we write $X1000$ for the case $(\String{X},1,0,0,0)$.

In fact, the cases  
$ G101, Z1011, X1011 $
do not appear in product \eqref{equ:mu}
and can take any value, in particular, zero.
In the cases $X1000$ and $X1001$, 
we have $\#\SymP=0$ by Lemma \ref{lemma:eqsum0}, so that
Identity \eqref{equ:mu} is true.
In the cases $(\String{G},0,1,0)$ and $(\nu, \eta)$ for 
$$
\begin{array}{lcl}
	\nu&=&\String{Z}, \String{X};\\
	\eta&=&(0,0,0,0), (0,1,0,0), (0,1,1,0), (1,1,0,0), \\
		 & &(0,0,0,1), (0,1,0,1), (0,1,1,1), (1,1,0,1),
\end{array}
$$
Lemma \ref{lemma:eq0} implies that $\#\SymP=0$. 
Hence,
Identity \eqref{equ:mu} is true.
%
All other cases are proved as follows.

\medskip

The evaluations of product \eqref{equ:mu} are explained in Section \ref{sec:CountType}, see Examples \ref{example1}--\ref{example4}. The factors $\mu_0, \mu_1,\ldots, \mu_{d-1}$ are obtained
	at the same time by proceeding with the substeps $3(w)$, $3(a)$, $3(e)$, $3(d)$, $3(b)$. In fact, we can evaluate each sole factor $\mu_i$ without keeping in mind the others. 
	For this purpose, we extract all biletters such that either its top letter is in $A_i$ or its bottom letter is $a_{d-i-1}$ in the first two substeps $3(w)$ and $3(a)$.  

Again, consider $i=1$ and $\SymP=\SymJ_{5n+2, 5n+1}^{dcbbaa}$.
We extract all biletters such that either its top letter is in $\{1,6,11,16, \ldots\}$ or its bottom letter is $d$ in the first two substeps $3(w)$ and $3(a)$ 
of Example \ref{example2}.  We have 
\begin{equation*}
\mathfrak{J}_{5n+2,5n+1}^{dcbbaa}
\overset{w}{=}
\left(
	\block{ \ch{\tilde5}}{\ch{d}}
	\bigg|
	\block{ \ch{\tilde1}\ch6\cch{1\spred1}\cch{1\spred6}}{\ch{c}\ch{d}\cch{d}\cch{\underline{a}}}
	\bigg|
	?
	\bigg|
	?
	\bigg|
	?
 	\right)	
\overset{a}{=}
\left(
	\block{ \ch{\tilde5}}{\ch{d}}
	\bigg|
	\block{ \ch{\tilde1}\ch6\cch{1\spred1}\cch{1\spred6}}{\ch{c}\ch{d}\cch{d}\cch{\underline{a}}}
	\bigg|
	?
	\bigg|
	\block{ \cch{1\spred8}}{\cch{\underline{1\spred8}}}
	\bigg|
	?
 	\right)	,
\end{equation*}
and
\begin{equation*}
\mathfrak{J}_{5n+2,5n+1}^{dcbbaa}
\overset{e}{=}
   \left(
	\block{ \cch{\tilde5}}{\cch{\underline{1\spred9}}}
	\bigg|
	\block{ \ch{\tilde1}\ch6\cch{1\spred1}\cch{1\spred6}}{\ch{\underline{d}}\ch{d}\cch{d}\cch{\underline{1\spred8}}}
	\bigg|
	?
	\bigg|
	\block{ \cch{1\spred8}}{\cch{\underline{b}}}
	\bigg|
	?
 	\right)
\overset{d}{=}\ 
	?\!\left(	\block{ \ch{\tilde1}\ch6\cch{1\spred1}\cch{1\spred6}}{\ch{\underline{d}}\ch{d}\cch{d}\cch{\underline{1\spred8}}}\right)
?
\ \overset{b}{=} \ 
? X_n ? ? ?.
\end{equation*}
It means that $\mu_1=X_n=\bar X_n$.
On the other hand, this case corresponds to the tuple $(\String{Z},1,0,0,0)$ that takes the value $\bar X_n$, as shown in Example~\ref{examplePsi2}.

In the sequel, 
$i_0=i,\, i_1=d+i,\, i_2=2d+i, \ldots, i_{n-1}=(n-1)d+i,\, i_n=nd+i$ are integers from $A_i$. Let $j=d-i-1,\, j_n=dn+j$.
We prove \eqref{equ:mu} case by case using the method described in the above example. Without loss of generality, the proof is illustrated for $n=3$.
\begin{align*}
G000 &:
\left(
\begin{array}{ccc|c}
i_0 & \tilde{\imath}_1& i_2 & ? \\
a_j &    ?            & a_j & a_j
\end{array}
\right)
\overset{e}{=}
\left(
\begin{array}{ccc}
i_0 & \tilde{\imath}_1 &i_2\\
a_j &\underline{a_j} &a_j
\end{array}
\right)
\overset{b}{=}\bar{X}_{n},\\
G001&:
\begin{pmatrix}i_0&i_1&i_2\\a_j&a_j&a_j\end{pmatrix}
\overset{b}{=}\bar{Y}_{n}, \\
G011&:
\left(
\begin{array}{c c c|c}
i_0&i_1&i_2&?\\a_j&a_j&a_j&a_j
\end{array}
\right)\\
&\overset{a}{=}
\left(
\begin{array}{c c c|c}
i_0&i_1&i_2&?\\a_j&a_j&a_j&a_j
\end{array}
\right)
\begin{pmatrix}i_{n}\\\underline{?}\end{pmatrix}
\overset{e}{=}
\begin{pmatrix}i_0&i_1&i_2&i_{n}\\a_j&a_j&a_j&\underline{a_j}\end{pmatrix}
\overset{b}{=}\bar{Z}_{n+1},\\
G100&:
\begin{pmatrix}i_0&\tilde{\imath}_1&i_2&i_3\\a_j&?&a_j&a_j\end{pmatrix}\\
&\overset{a}{=}
\begin{pmatrix}i_0&\tilde{\imath}_1&i_2&i_3\\a_j&?&a_j&a_j\end{pmatrix}
\begin{pmatrix}?\\\underline{j_n}\end{pmatrix}
\overset{e}{=}
\begin{pmatrix}i_0&\tilde{\imath}_1&i_2&i_3\\a_j&j_n&a_j&a_j\end{pmatrix}
\overset{b}{=}
\bar{Z}_{n+1},\\
G110&:
\left(
\begin{array}{cccc|c}
	i_0 &\tilde{\imath}_1 &i_2 &i_3 &? \\
	a_j &? &a_j &a_j & a_j
\end{array}
\right)
\overset{e}{=}
\left(
\begin{array}{cccc}
	i_0 &\tilde{\imath}_1 &i_2 &i_3  \\
	a_j &\underline {a_j} &a_j &a_j 
\end{array}
\right)
\overset{b}{=}\bar{X}_{n+1},\\
G111&:
\begin{pmatrix}i_0&i_1&i_2 & i_3\\a_j&a_j&a_j&a_j\end{pmatrix}
\overset{b}{=}\bar{Y}_{n+1},\\
  Z0010&:
  \left(
  \begin{array}{ccc|c}
		i_0&i_1&i_{n-1}&?\\a_j&a_j&?&a_j
  \end{array}
  \right)
\overset{e}{=}
	\begin{pmatrix}i_0&i_1&i_{n-1}\\a_j&a_j&\underline{a_j}\end{pmatrix}
\overset{b}{=}\bar{Z}_{n},\\
  Z1110&:
  \left(
  \begin{array}{cccc|c}
    i_0&i_1&i_2& i_n&?\\a_j&a_j&a_j&?&a_j
  \end{array}
  \right)
\overset{e}{=}
  \begin{pmatrix}i_0&i_1&i_2 &i_n\\a_j&a_j&a_j&\underline{a_j}\end{pmatrix}
\overset{b}{=}\bar{Z}_{n+1},\\
  Z0011&:
		\begin{pmatrix}i_0&i_1&i_{n-1}\\a_j&a_j&\underline{a_j}\end{pmatrix}
\overset{b}{=}\bar{Z}_{n},\\
  Z1111&:
   \begin{pmatrix}i_0&i_1&i_2 &i_n\\a_j&a_j&a_j&\underline{a_j}\end{pmatrix}
\overset{b}{=}\bar{Z}_{n+1},\\
  Z1000&:
  \left(
  \begin{array}{cccc|c}
    i_0&\tilde{\imath}_1&i_2 &i_n & ?\\
    a_j&?&a_j&\underline{?}& a_j
  \end{array}
  \right)\\
&\overset{a}{=}
  \left(
  \begin{array}{cccc|c}
    i_0&\tilde{\imath}_1&i_2 &i_n & ?\\
    a_j&?&a_j&\underline{?}& a_j
  \end{array}
  \right)
  \begin{pmatrix}? \\\underline{j_n}\end{pmatrix}
\overset{e}{=}
  \begin{pmatrix}
    i_0&\tilde{\imath}_1&i_2 &i_n \\
    a_j&\underline{a_j}&a_j&\underline{j_n}
  \end{pmatrix}
\overset{b}{=}\bar{X}_{n}, \\
  Z1001&:
  \left(
  \begin{array}{cccc}
    i_0&\tilde{\imath}_1&i_2 &i_n \\
    a_j&?&a_j&\underline{a_j}
  \end{array}
  \right)\\
&\overset{a}{=}
  \left(
  \begin{array}{cccc}
    i_0&\tilde{\imath}_1&i_2 &i_n \\
    a_j&?&a_j&\underline{a_j}
  \end{array}
  \right)
  \begin{pmatrix}? \\\underline{j_n}\end{pmatrix}
\overset{e}{=}
  \begin{pmatrix}
    i_0&\tilde{\imath}_1&i_2 &i_n \\
    a_j&\underline{a_j}&a_j&\underline{j_n}
  \end{pmatrix}
\overset{b}{=}\bar{X}_{n},\\
  Z1010&:
  \left(
  \begin{array}{cccc}
    i_0&i_1&i_2 &i_n \\
    a_j&a_j&a_j&\underline{?}
  \end{array}
  \right)\\
&\overset{a}{=}
  \left(
  \begin{array}{cccc}
    i_0&i_1&i_2 &i_n \\
    a_j&a_j&a_j&\underline{?}
  \end{array}
  \right)
  \begin{pmatrix}? \\\underline{j_n}\end{pmatrix}
\overset{e}{=}
  \begin{pmatrix}
    i_0&i_1&i_2 &i_n \\
    a_j&a_j&a_j&\underline{j_n}
  \end{pmatrix}
\overset{b}{=}\bar{Y}_{n},\\
  X0010&:
  \left(
  \begin{array}{ccc|c}
    i_0&\tilde{\imath}_1&i_2&?\\a_j&?&a_j&a_j
  \end{array}
  \right)
\overset{e}{=}
  \begin{pmatrix}i_0&\tilde{\imath}_1&i_2\\a_j&\underline{a_j}&a_j\end{pmatrix}
\overset{b}{=}\bar{X}_{n},\\
  X0011&:
  \left(
  \begin{array}{ccc}
    i_0&\tilde{\imath}_1&i_2\\a_j&\underline{a_j}&a_j
  \end{array}
  \right)
\overset{b}{=}\bar{X}_{n},\\
  X1110&:
  \left(
  \begin{array}{cccc|c}
    i_0&\tilde{\imath}_1&i_2&i_3&?\\a_j&?&a_j&a_j&a_j
  \end{array}
  \right)
\overset{e}{=}
  \begin{pmatrix}i_0&\tilde{\imath}_1&i_2&i_3\\a_j&\underline{a_j}&a_j&a_j\end{pmatrix}
\overset{b}{=}\bar{X}_{n+1},\\
  X1111&:
  \begin{pmatrix}
    i_0&\tilde{\imath}_1&i_2 & i_3\\a_j&\underline{a_j}&a_j&a_j
  \end{pmatrix}
\overset{b}{=}\bar{X}_{n+1},\\
  X1010&:
  \begin{pmatrix}
    i_0&\tilde{\imath}_1&i_2 &i_3 \\
    a_j&\underline{?}&a_j&a_j
  \end{pmatrix}\\
&\overset{a}{=}
  \begin{pmatrix}
    i_0&\tilde{\imath}_1&i_2 &i_3 \\
    a_j&\underline{?}&a_j&a_j
  \end{pmatrix}
  \begin{pmatrix}? \\\underline{j_n}\end{pmatrix}
\overset{e}{=}
  \begin{pmatrix}
    i_0&\tilde{\imath}_1&i_2 &i_3 \\
    a_j&\underline{j_n}&a_j&a_j
  \end{pmatrix}
\overset{b}{=}\bar{Z}_{n+1}.\qedhere
\end{align*}

\end{proof}


\section{Implementation and outputs}
Our program {\tt Apwen.py} is an
implementation of Algorithms 1 and~2 in {\tt Python}.
The proofs of Lemmas \ref{lemma:main3XYZ}, \ref{lemma:main3UVW} and \ref{lemma:main5} are achieved by 
the following Outputs 1--3 of the program {\tt Apwen.py} respectively.
For simplicity, the expression $n+1$ is reproduced by letter $m$.
Thus, {\tt "Y(3n+1) = Vn Wm"} means the following recurrence relation
$$
Y_{3n+1} = V_n W_{n+1},
$$
which appeared in Lemma \ref{lemma:main3XYZ}.
The calculations made in Examples \ref{example1}--\ref{example4} in Section \ref{sec:CountType} 
can be found in Output 3, types {\tt 168 adbca, 219 dcbbaa, 145 adcbac, 213 edcaab} respectively.

\medskip

\hrule\smallskip
\noindent
{\bf Output 1} {\tt "python Apwen.py 3"}
\smallskip
\hrule\smallskip
{\footnotesize
\begin{verbatim}
v= [1, -1, -1]     direction = XYZ -> UVW
P= [1]
Q= [2]
J= [0, 3, 5, 6, 8, 9, 12, 14, 15, 18, 21, 23, 24, 27, ...]
K= [1, 2, 4, 7, 10, 11, 13, 16, 17, 19, 20, 22, 25,  ...]

 k : 3N+0 
 1 cbac: [Un:X0011] [Vn:G001] [Vn:G001]
 2 ccab: [Un:X0010] [Un:G000] [Vn:G001]
 3 ccba: [Un:X0010] [Un:G000] [Un:G000]
 k : 3N+1 
 4 abbc: [Un:G000] [Un:X0010] [Un:G000]
 5 cbab: [Vn:G001] [Un:X0011] [Vn:G001]
 6 cbba: [Vn:G001] [Un:X0010] [Un:G000]
 k : 3N+2 
 7 abac: [Un:G000] [Vn:G001] [Un:X0010]
 8 acab: [Un:G000] [Un:G000] [Un:X0010]
 9 cbaa: [Vn:G001] [Vn:G001] [Un:X0011]
X(3n+0) =  Un

 k : 3N+0 
 10 cbaa: [Wm:X1010] [Vn:G001] [Wm:G011]
 k : 3N+1 
 11 abab: [Wm:G100] [Un:X0011] [Wm:G011]
 k : 3N+2 
X(3n+1) =  Un Wm + Vn Wm

 k : 3N+0 
 12 cbaa: [Wm:X1010] [Vm:G111] [Wm:G011]
 k : 3N+1 
 13 abab: [Wm:G100] [Um:X1111] [Wm:G011]
 k : 3N+2 
X(3n+2) =  Um Wm + Vm Wm

 14 acb: [Un:G000] [Un:G000] [Un:G000]
 15 cba: [Vn:G001] [Vn:G001] [Vn:G001]
Y(3n+0) =  Un + Vn

 16 aba: [Wm:G100] [Vn:G001] [Wm:G011]
Y(3n+1) =  Vn Wm

 17 aba: [Wm:G100] [Vm:G111] [Wm:G011]
Y(3n+2) =  Vm Wm

 18 abac: [Un:G000] [Vn:G001] [Wn:Z0010]
 19 acab: [Un:G000] [Un:G000] [Wn:Z0010]
 20 cbaa: [Vn:G001] [Vn:G001] [Wn:Z0011]
Z(3n+0) =  Un Vn Wn + Un Wn + Vn Wn

 21 abac: [Un:Z1001] [Vn:G001] [Wm:G011]
 22 acab: [Un:Z1000] [Un:G000] [Wm:G011]
 23 cbaa: [Vn:Z1010] [Vn:G001] [Wm:G011]
Z(3n+1) =  Un Vn Wm + Un Wm + Vn Wm

 24 abab: [Wm:G100] [Wm:Z1111] [Wm:G011]
Z(3n+2) =  Wm
\end{verbatim}
}
\medskip

\goodbreak
\hrule\smallskip
\noindent
{\bf Output 2} {\tt "python Apwen.py -3"}
\smallskip
\hrule\smallskip
\nobreak
{\footnotesize
\begin{verbatim}
v= [1, -1, -1]     direction = UVW -> XYZ
P= [2]
Q= [1]
J= [1, 2, 4, 7, 10, 11, 13, 16, 17, 19, 20, 22, 25, ...]
K= [0, 3, 5, 6, 8, 9, 12, 14, 15, 18, 21, 23, 24, 27, ...]

 k : 3N+0 
 1 cacb: [Xn:X0010] [Xn:G000] [Xn:G000]
 2 cbac: [Xn:X0011] [Yn:G001] [Yn:G001]
 3 cbca: [Xn:X0010] [Yn:G001] [Xn:G000]
 k : 3N+1 
 4 bbac: [Xn:G000] [Xn:X0010] [Yn:G001]
 5 bbca: [Xn:G000] [Xn:X0010] [Xn:G000]
 6 cbab: [Yn:G001] [Xn:X0011] [Yn:G001]
 k : 3N+2 
 7 baac: [Xn:G000] [Xn:G000] [Xn:X0010]
 8 caab: [Yn:G001] [Xn:G000] [Xn:X0010]
 9 cbaa: [Yn:G001] [Yn:G001] [Xn:X0011]
U(3n+0) =  Xn

 k : 3N+0 
 10 caab: [Zm:X1010] [Xn:G000] [Zm:G011]
 11 cbaa: [Zm:X1010] [Yn:G001] [Zm:G011]
 k : 3N+1 
 12 bbaa: [Zm:G100] [Xn:X0010] [Zm:G011]
 k : 3N+2 
U(3n+1) =  Yn Zm

 k : 3N+0 
 13 caab: [Zm:X1010] [Xm:G110] [Zm:G011]
 14 cbaa: [Zm:X1010] [Ym:G111] [Zm:G011]
 k : 3N+1 
 15 bbaa: [Zm:G100] [Xm:X1110] [Zm:G011]
 k : 3N+2 
U(3n+2) =  Ym Zm

 16 bac: [Xn:G000] [Xn:G000] [Xn:G000]
 17 cba: [Yn:G001] [Yn:G001] [Yn:G001]
V(3n+0) =  Xn + Yn

 18 baa: [Zm:G100] [Xn:G000] [Zm:G011]
V(3n+1) =  Xn Zm

 19 baa: [Zm:G100] [Xm:G110] [Zm:G011]
V(3n+2) =  Xm Zm

 20 baac: [Xn:G000] [Xn:G000] [Zn:Z0010]
 21 caab: [Yn:G001] [Xn:G000] [Zn:Z0010]
 22 cbaa: [Yn:G001] [Yn:G001] [Zn:Z0011]
W(3n+0) =  Xn Yn Zn + Xn Zn + Yn Zn

 23 baac: [Xn:Z1001] [Xn:G000] [Zm:G011]
 24 caab: [Yn:Z1010] [Xn:G000] [Zm:G011]
 25 cbaa: [Yn:Z1010] [Yn:G001] [Zm:G011]
W(3n+1) =  Xn Yn Zm + Xn Zm + Yn Zm

 26 bbaa: [Zm:G100] [Zm:Z1110] [Zm:G011]
W(3n+2) =  Zm

\end{verbatim}
}

\goodbreak
\hrule\smallskip
\noindent
{\bf Output 3} {\tt "python Apwen.py 5"} (extract)
\smallskip
\hrule\smallskip
\nobreak
{\footnotesize
\begin{verbatim}
v= [1, -1, -1, -1, 1]    direction = XYZ -> XYZ
P= [1, 4]
Q= [2, 3]
J= [0, 3, 4, 5, 8, 10, 13, 15, 18, 19, 20, 23, 24, 25, 28, 29]
K= [1, 2, 6, 7, 9, 11, 12, 14, 16, 17, 21, 22, 26, 27]

 ...
 k : 5N+2 
 144 accbad: [Zm:G100] [Xm:G110] [Xm:X1110] [Ym:G111] [Zm:G011]
 145 adcbac: [Zm:G100] [Ym:G111] [Xm:X1111] [Ym:G111] [Zm:G011]
 146 adccab: [Zm:G100] [Ym:G111] [Xm:X1110] [Xm:G110] [Zm:G011]
 147 dccaab: [Zm:G100] [Xm:G110] [Xm:X1110] [Xm:G110] [Zm:G011]
 148 dccbaa: [Zm:G100] [Xm:G110] [Xm:X1110] [Ym:G111] [Zm:G011]
 k : 5N+3 
 149 acbbad: [Zm:G100] [Xm:G110] [Xm:G110] [Xm:X1110] [Zm:G011]
 150 acdbab: [Zm:G100] [Xm:G110] [Xm:G110] [Xm:X1111] [Zm:G011]
 151 adbbac: [Zm:G100] [Ym:G111] [Xm:G110] [Xm:X1110] [Zm:G011]
 152 adcbab: [Zm:G100] [Ym:G111] [Ym:G111] [Xm:X1111] [Zm:G011]
 153 dcbbaa: [Zm:G100] [Xm:G110] [Xm:G110] [Xm:X1110] [Zm:G011]
 k : 5N+4 
X(5n+4) =  Ym Zm

...
 167 acdba: [Zm:G100] [Xn:G000] [Xn:G000] [Yn:G001] [Zm:G011]
 168 adbca: [Zm:G100] [Yn:G001] [Xn:G000] [Xn:G000] [Zm:G011]
 169 adcba: [Zm:G100] [Yn:G001] [Yn:G001] [Yn:G001] [Zm:G011]
 170 dcbaa: [Zm:G100] [Xn:G000] [Xn:G000] [Xn:G000] [Zm:G011]
Y(5n+1) =  Xn Zm + Yn Zm

...
 212 edbcaa: [Yn:Z1010] [Yn:G001] [Xn:G000] [Xn:G000] [Zm:G011]
 213 edcaab: [Yn:Z1010] [Yn:G001] [Yn:G001] [Xn:G000] [Zm:G011]
 214 edcbaa: [Yn:Z1010] [Yn:G001] [Yn:G001] [Yn:G001] [Zm:G011]
Z(5n+1) =  Xn Yn Zm + Xn Zm + Yn Zm

 215 acbbad: [Zm:G100] [Xn:Z1001] [Xn:G000] [Zm:G011] [Zm:G011]
 216 acdbab: [Zm:G100] [Xn:Z1000] [Xn:G000] [Zm:G011] [Zm:G011]
 217 adbbac: [Zm:G100] [Yn:Z1010] [Xn:G000] [Zm:G011] [Zm:G011]
 218 adcbab: [Zm:G100] [Yn:Z1010] [Yn:G001] [Zm:G011] [Zm:G011]
 219 dcbbaa: [Zm:G100] [Xn:Z1000] [Xn:G000] [Zm:G011] [Zm:G011]
Z(5n+2) =  Xn Yn Zm + Xn Zm + Yn Zm

...
\end{verbatim}
}
\smallskip
\hrule\medskip

\medskip

The proof of that $F_{13}$ is Apwenian takes 11 hours by using 
the program {\tt Apwen.py} on a modern personal computer. 
For proving that $F_{17a}$ and $F_{17b}$ are Apwenian, it was necessary to 
rewrite the program in the {\tt C} language with some optimizations.
The running times of the two programs are reproduced in the following table:
\begin{equation*}
\begin{array}{|c|c|c|c|c|c|c|}
	\hline
	{\bf f}   & F_3  & F_5  & F_{11}  & F_{13}  & F_{17a}, F_{17b}  & F_{19}\\
	\hline
	{\tt Python}&<1 s & <1 s & 11 m & 11 h & \infty & \infty\\
	{\tt C}     &<1 s & <1 s & 16 s & 29 m & 7\, \text{days} \times 24\, \text{CPUs} & \infty \\
	\hline
\end{array}
\end{equation*}

\goodbreak

\bibliographystyle{plain}
\bibliography{\jobname}

\end{document}